\documentclass[reqno, 12pt]{amsart}
\usepackage{array}
\usepackage{amsmath}
\usepackage{amsfonts}
\usepackage{amssymb}
\usepackage{enumitem}
\usepackage{amsthm}
\usepackage{amsmath, amscd}
\usepackage{mathtools}
\usepackage{tensor}
\usepackage{bm}
\usepackage{xy}
\usepackage{tikz-cd}
\xyoption{all}
\usepackage{color}
\usepackage{marvosym}
\usepackage{amsbsy}
\usepackage{hyperref}
\usepackage{tikz}
\usetikzlibrary{matrix,arrows,backgrounds}
\usepackage{leftidx}
\usepackage{MnSymbol}
\usepackage{anyfontsize}
\usepackage{pict2e}

\makeatletter
\newcommand{\adjunction}[4]{%
  #1\colon #2%
  \mathrel{\vcenter{%
    \offinterlineskip\m@th
    \ialign{%
      \hfil$##$\hfil\cr
      \longrightharpoonup\cr
      \noalign{\kern-.3ex}
      \smallbot\cr
      \longleftharpoondown\cr
    }%
  }}%
  #3 \noloc #4%
}
\newcommand{\longrightharpoonup}{\relbar\joinrel\rightharpoonup}
\newcommand{\longleftharpoondown}{\leftharpoondown\joinrel\relbar}
\newcommand\noloc{%
  \nobreak
  \mspace{6mu plus 1mu}
  {:}
  \nonscript\mkern-\thinmuskip
  \mathpunct{}
  \mspace{2mu}
}
\newcommand{\smallbot}{%
  \begingroup\setlength\unitlength{.15em}%
  \begin{picture}(1,1)
  \roundcap
  \polyline(0,0)(1,0)
  \polyline(0.5,0)(0.5,1)
  \end{picture}%
  \endgroup
}
\makeatother

\newtheorem*{introtheorem}{Theorem}
\newtheorem*{introconstruction}{Construction}

\newtheorem{theorem}{Theorem}[section]

\newtheorem{lemma}[theorem]{Lemma}
\newtheorem{proposition}[theorem]{Proposition}
\newtheorem{corollary}[theorem]{Corollary}

\newtheorem{question}[theorem]{Question}

\theoremstyle{definition}
\newtheorem{definition}[theorem]{Definition}
\newtheorem{remark}[theorem]{Remark}

\newtheorem{example}[theorem]{Example}

\newcommand{\op}[1]{\operatorname{#1}}

\newcommand{\leftexp}[2]{{\vphantom{#2}}^{#1}{#2}}

\newcommand{\weezer}{\leftexp{=}{\kern-0.23em\operatorname{W}}^{\kern-0.21em =}}

\newcommand{\newterm}{\textsf}

\def\N{\op{\mathbb{N}}}
\def\Z{\op{\mathbb{Z}}}

\def\O{\op{\mathcal{O}}}
\def\A{\\op{mathcal{A}}}

\def\tand{\text{ and } }
\def\A{\mathop{\mathcal{A}}}

\def\Gm{\mathbb{G}_m}
\def\gm{\mathbb{G}_m}

\def\A{\mathbb A}

\renewcommand{\labelenumi}{\theenumi}

\title[Kernels from compactifications]{Kernels from compactifications}

\author[Ballard]{Matthew R Ballard}
\address{
  \begin{tabular}{l}
   Matthew Robert Ballard \\
   \hspace{.1in} University of South Carolina\\
   \hspace{.1in} Department of Mathematics \\
   \hspace{.1in} Columbia, SC, USA  \\
   \hspace{.1in} Email: {\bf ballard@math.sc.edu} \\
  \end{tabular}
}

\author[Diemer]{Colin Diemer}
\address{
  \begin{tabular}{l}
   Colin Diemer \\
   \hspace{.1in} Institut des Haute \'Etudes Scientifiques \\
   \hspace{.1in} Bures sur Yvette, France \\
   \hspace{.1in} Laboratory of Mirror Symmetry \\
   \hspace{.1in} NRU HSE\\
   \hspace{.1in} Email: {\bf diemer@ihes.fr} \\
   \end{tabular}
}

\author[Favero]{David Favero}
\address{
  \begin{tabular}{l}
   David Favero \\
   \hspace{.1in} University of Alberta \\
   \hspace{.1in} Department of Mathematical and Statistical Sciences \\
   \hspace{.1in} Central Academic Building 632, Edmonton, AB, Canada T6G 2C7 \\
   \hspace{.1in} Korean Institute for Advanced Study \\
   \hspace{.1in} 85 Hoegiro, Dongdaemun-gu, Seoul, Republic of Korea 02455 \\
   \hspace{.1in} Email: {\bf favero@ualberta.ca} \\
  \end{tabular}
}


\numberwithin{equation}{section}
\numberwithin{theorem}{subsection}

\begin{document}

\renewcommand{\labelenumi}{\emph{\alph{enumi})}}

\begin{abstract}
Associated to any affine scheme with a $\Gm$-action, we provide a Bousfield colocalization on the equivariant derived category $\op{D}(\op{Mod}^{\Gm} R)$ by constructing an idempotent integral kernel using homotopical methods. This endows the equivariant derived category with a canonical semi-orthogonal decomposition.   As a special case, we demonstrate that grade-restriction windows appear as a consequence of this construction, giving a new proof of wall-crossing equivalences which works over an arbitrary base. The construction globalizes to yield interesting integral transforms associated to $D$-flips.
\end{abstract}

\maketitle

\section{Introduction} \label{section: Introduction}

A central question in the study of derived categories of coherent sheaves of algebraic varieties is their relationship with birational geometry. Historically, such investigations originated with Orlov's construction of a semi-orthogonal decomposition associated to a blow-up \cite{Orl92}, as well as Bondal and Orlov's derived equivalences induced by certain elementary flops \cite{BO}. Much recent effort has since centered around Bondal and Orlov's conjecture that flops in general induce derived equivalences as well as Kawamata's related conjecture \cite{Kaw} that $K$-equivalent varieties have equivalent derived categories. Perhaps the most striking result in this direction is Bridgeland's construction of a derived equivalence for any threefold flop \cite{Bridge}. A more recent line of inquiry is the description of the equivariant derived categories of geometric invariant theory (GIT) quotients via so-called grade restriction windows, see e.g. \cite{Seg, BFK, HHP, HL, Ball, SvdBNon, HLSam}. These methods sometimes give equivalences or semi-orthogonal decompositions associated to birational maps by viewing the maps as GIT wall-crossings. 

Nevertheless, a quick survey of the subject will convince an observer that there is not an agreed upon uniform approach to producing the functors expected by the $K$-equivalence conjecture.  For example, Bridgeland's techniques require that the flop come from a small contraction over a base of relative dimension one, which limits the applicability to higher dimensional flops. Various families of explicit flops have been considered; notably Namikawa and Kawamata's study of Mukai flops \cite{N1,N2,KMukai} and Cautis, Kamnitzer, and Licata's study of the the stratified Mukai flop \cite{CKL}. In these stratified examples a derived equivalence has indeed been observed, but only via fine-tuned choices of explicit Fourier-Mukai kernels. The grade restriction window techniques mentioned above have so far been most effective only for so-called elementary wall-crossings, or when the action is specialized to be quasi-symmetric in the language of \cite{SvdBNon}. 

In particular, there does not presently appear to be a consensus in the literature for approaching the following problem: given an arbitrary birational map $X\dashrightarrow Y$ of Mori theoretic origin, provide a uniform method of producing a homologically well-behaved functor between $\op{D}^b(\op{coh} X) $ and $\op{D}^b(\op{coh}Y)$. In other words, how to  systematically produce a Fourier-Mukai kernel object $P\in\op{D}^b(\op{coh}X\times Y)$ consistent with the expectations of the Bondal-Orlov and Kawamata conjectures? We summarize the main construction of this paper as follows.

\begin{introconstruction}
Let $Y$ be a scheme with a trivial $\gm$-action, $\mathcal A$ be a quasi-coherent sheaf of $\Z$-graded $\O_Y$-algebras, and set $Z := \underline{\op{Spec}}_{Y} \, \mathcal A$.

\begin{itemize}

\item Form a $\Z^2$-graded sheaf $Q_{\op{der}}(\mathcal A)$ of $\mathcal A \otimes_{\O_Y} \mathcal A$-algebras by deriving a certain partial compactification  \cite{Drinfeld} of the action groupoid.

\item Realize $Q_{\op{der}}(\mathcal A)$ as an object of $\op{D}(\op{Qcoh}^{\Gm^2} Z \times Z)$.

\item Restrict $Q_{\op{der}}(\mathcal A)$ to open sets to get an equivariant Fourier-Mukai transform
\[
\Phi_{Q^{wc}_{\op{der}}(\mathcal A)} : \op{D}(\op{Qcoh}^{\Gm} U^+ ) \to  \op{D}(\op{Qcoh}^{\Gm} U^- ).
\] where $U^{\pm}$ are the corresponding semi-stable loci. 

\end{itemize}

\end{introconstruction}

A central case of interest is when $X \dashrightarrow Y$ is a flip relative to a divisor $D$ on $X$, i.e. a $D$-flip. An observation of Reid, see e.g. \cite{Thaddeus}, allows one to repackage the data of the $D$-flip as a scheme $Z$ affine over the contraction and carrying an action of $\Gm$, with $X$ and $Y$ the respective GIT quotients, so that the above construction applies. We ask in Question~\ref{main question} if $\Phi_{Q^{wc}_{\op{der}}(\mathcal A)}$ induces an equivalence for $D$-flops, a potential solution to  \cite[Conjecture 5.1]{Kaw}.
However, we presently content ourselves with the following results (see Proposition~\ref{proposition: S Q sod}, Proposition~\ref{proposition: global LQ is qcoh}, Theorem~\ref{theorem: smoothcasesummary}, and Corollary~\ref{corollary: affinewc} for more precise statements).

\begin{introtheorem} 
There is an object $S_{\op{der}} \in \op{D}(\op{Qcoh}^{\Gm^2} Z \times Z)$ and a semi-orthogonal decomposition
  \[
   \op{D}(\op{Qcoh}^{\Gm} Z) = \langle \op{Im } \Phi_{Q_{\op{der}(\mathcal A)}}, \op{Im } \Phi_{S_{\op{der}(\mathcal A)}} \rangle.
  \]
When $Z$ is smooth and affine then the image of $\op{Im } \Phi_{Q_{\op{der}(\mathcal A)}}$ over $U^+ \times Z$ is equal to the grade restriction window defined in \cite{Seg, BFK, HL}.  Hence, when $[U^+ / \gm]$ and $[U^- / \gm]$ are $K$-equivalent, $\Phi_{Q^{wc}_{\op{der}}(\mathcal A)}$ is an equivalence.
\end{introtheorem} 

In the smooth case, we need not derive our construction and simply denote this object by $Q$.  Here, the semi-orthogonal decomposition above comes from a certain idempotent property enjoyed by $Q$ which we call Property $\mathtt{P}$, see Definition ~\ref{definition: PropertyP}, which shows that $Q$ induces a Bousfield localization. We remark that when $Z$ is smooth, the proof given here is quite different than those articles as here we produce an explicit geometric kernel, prove functorial identities of that kernel, and deduce these results as corollaries.  The proof also works over an arbitrary base. The essential observation here is that the construction of $Q$ behaves well under strongly \'etale base change (see Proposition~\ref{prop: ready for Luna}) which allows us to reduce to the case of affine space using the Luna Slice Theorem.
 

%


When $Z$ is singular, it is not the case that the object $Q$ literally enjoys the idempotent Property $\mathtt{P}$ mentioned above. This is problematic as the case of singular affine varieties equipped with a $\Gm$-action is quite  important for the demands of birational geometry (even for the elementary Mukai flop, the corresponding space $Z$ is singular, for example). This is the reason that we must derive $Q$, i.e. promote $Q$ to a object in derived algebraic geometry. 

In Section ~\ref{section: simplicialQ} we observe that the functor $Q$ extends to a left Quillen functor on the category of graded simplicial rings, i.e. on derived affine schemes equipped with an action of $\Gm$. Theorem ~\ref{theorem: derived Q2=Q} shows that this derived variant $Q_{\op{der}}$ does indeed satisfy an analogue of the idempotent property, which we call Property $\mathtt{P}_{\op{der}}$, and so we still obtain a semi-orthogonal decomposition in analogy with the smooth case.  At an intuitive level, the failure of $Q$ to be well-behaved for singular spaces arises from the non-vanishing of some higher Tor's (see e.g. Lemma ~\ref{lemma: Q S bousfield triangle}), and the homotopical methods mentioned above allow us to  bypass this obstruction by encoding the higher Tor's in an intrinsic derived affine scheme. 

Our actual construction of $Q$ comes from the following geometric consideration: if an algebraic group $G$ acts on a scheme $Z$, we consider a space $\tilde{Z}$ which equivariantly extends the action and projection maps; see Definition ~\ref{definition: compactification}. Such a construction for $\Gm$-actions was already considered by Drinfeld \cite{Drinfeld}. To such data, one can always exhibit an associated faithful functor, see e.g. Proposition ~\ref{proposition: faithful}. Given Drinfeld's construction and the central role of $\Gm$-actions via birational cobordisms, we will likewise focus primarily on this case in this article, although the reader will find a discussion of more general group actions in Section ~\ref{section: compactifications}.  We also remark that at a technical level most of the results in the paper are formulated for affine varieties only, but in Section ~\ref{section: globalcase} we discuss how to associate sheaves to these constructions in order to extend our results to essentially arbitrary $D$-flips, as mentioned above. 

The paper is organized as follows. \begin{itemize}
\item In Section ~\ref{section: affine} we introduce our main object of study, $Q(R)$, and study its basic properties.  Subsection ~\ref{section: flops} goes on to show by example how this object arises in the study of flips and flops. 

\item In Section ~\ref{section: compactifications} we discuss the geometric interpretation of $Q$ in terms of compactifications of group actions; here we also introduce the basic criteria for fully-faithfulness and its relation to Bousfield localizations. 

\item Section ~\ref{section: smooth} treats the case where $\op{Spec}R$ is smooth and exhibits the essential image as a grade restriction window. 

\item Section ~\ref{section: simplicialQ} then treats the general case by deriving the construction of $Q$ itself. Subsection ~\ref{section: derivedsmoothcase} shows that when $\op{Spec}R$ is smooth, these derived replacements effectively trivialize down to the methods of Section ~\ref{section: smooth}. Section ~\ref{section: globalcase} discusses the globalization process for attaching a kernel object to a general $D$-flip. 

\item  In a brief appendix we comment on a small, but important, distinction between our $Q(R)$ and the related constructions in Drinfeld's article ~\cite{Drinfeld}. 

\end{itemize} 

\noindent{\em Acknowledgements:} This paper began during a meeting with Ludmil Katzarkov and Maxim Kontsevich at the Institut des Hautes \'Etudes Scientifiques.  We thank them for their continual support and numerous insights as the project evolved.  We also benefited from an illuminating conversation with Dan Halpern-Leistner, which in particular drew our attention to the relevant work of Drinfeld. We thank him for his interest. In addition, D. Favero would like to thank Bumsig Kim for helpful conversations exploring non-abelian cases, and M. Ballard would like to thank L. Borisov and D. Orlov for discussions during his visit at the Institute for Advanced Study.
M. Ballard was partially supported by NSF Standard Grant DMS-1501813 and is appreciative of the Institute for Advanced Study for furnishing a wonderful work environment during his membership. Finally, he would like to express appreciation to both Bob Ballards in his life.  D. Favero was partially funded by the Natural Sciences and Engineering Research Council of Canada and Canada Research Chair Program under NSERC RGPIN 04596 and CRC TIER2 229953 and also thanks the  Korean Institute for Advanced Study for their hospitality.  C. Diemer is supported by a Simons Collaborative Grant postdoctoral fellowship and acknowledges support by the Laboratory of Mirror Symmetry NRU HSE, RF Government grant, ag. No. 14.641.31.0001, and would like to thank everyone at the IHES for an ideal working environment.

\section{Affine constructions}\label{section: affine}

\subsection{The functor} \label{section: the functor} Let $k$ be a fixed commutative ring. If $R$ is a $\mathbb{Z}$-graded $k$-algebra, we consider the two maps \begin{align}
\pi &: R  \to  R\otimes_kk[u,u^{-1}]\cong R[u,u^{-1}]  \label{equation: projection} \\
\sigma &:   R  \to R\otimes_kk[u,u^{-1}]\cong R[u,u^{-1}]  \label{equation: action} \end{align}
where $\pi(r) = r$ is the identity and $\sigma$ is the (co-)action map determined by $\sigma (r) = ru^{\op{deg}(r)}$ when $r$ is homogenous. In terms of affine schemes these maps respectively correspond to the projection and action maps 
\begin{equation}\label{equation: actionandprojection}
\begin{tikzcd}
\Gm \times_k \textrm{Spec}R \ar[r, shift left, "\widehat{\pi}"] \ar[r, shift right, swap, "\widehat{\sigma}"] & \textrm{Spec}R.
\end{tikzcd}
\end{equation}
Throughout this article, we will let $\mathsf{CR}^{\Gm}_k$ denote the category of finitely generated $\mathbb{Z}$-graded $k$-algebras, i.e. the opposite category of the category of affine $k$-schemes which are equipped with $\Gm$-actions. We will often refer to objects of $\mathsf{CR}_k^{\Gm}$ as rings with a $\Gm$-action. Likewise, if an algebraic group $G$ acts on $\op{Spec}R$ we will say that $R$ has a $G$-action.

When $R$ is an object of $\mathsf{CR}_k^{\Gm}$ the product space $\Gm\times_k\textrm{Spec}R$ admits a $\Gm^2$-action whose structure is given in the following lemma, the proof of which is elementary and which we omit. 
\begin{lemma}\label{lemma: basic action}
 The ring $R[u,u^{-1}]$ carries a natural $\Gm^2$-action 
 \begin{displaymath}
  \widetilde{\sigma} : R \to R[u,u^{-1},v_1,v_1^{-1},v_2,v_2^{-1}]
 \end{displaymath}
 uniquely determined by
 \begin{align*}
  \widetilde{\sigma}(\pi(r)) & = \sigma_1(r) \in R[u,u^{-1},v_1,v_1^{-1},v_2,v_2^{-1}] \\
  \widetilde{\sigma}(\sigma(r)) & = \sigma_2(r) \in R[u,u^{-1},v_1,v_1^{-1},v_2,v_2^{-1}]
 \end{align*}
 where $\sigma_i : R \to R[v_i,v_i^{-1}]$ for $i=1,2$ are both identified with $\sigma$ from Equation ~\eqref{equation: action}.
 
 Moreover, the map $\hat\pi : \Gm\times\op{Spec}R\to \op{Spec}R$ becomes equivariant with respect to the projection $\pi_1: \Gm^2 \to \Gm$ onto the first factor, and the map $\hat\sigma$ becomes equivariant with respect to the projection $\pi_2: \Gm^2 \to \Gm$ onto the second factor. 
\end{lemma}

\begin{remark}\label{remark: gradings} 
The $\Gm^2$-action above is equivalent to saying that $R[u,u^{-1}]$ has a $\Z^2$-grading with the degree of $r\in R\subset R[u,u^{-1}]$ with $r$ homogenous being $( \op{deg}r, 0 )$, the degree of $\sigma(r)$ being $(0,\op{deg}r)$, and the degree of $u$ being $(-1,1)$. 
\end{remark}

\noindent Given a map $R\to S$ in $\mathsf{CR}_k^{\Gm}$, one obtains the obvious induced map $R[u,u^{-1}]\to S[u,u^{-1}]$. Here we regard $R[u,u^{-1}]$ and $S[u,u^{-1}]$ as objects of $\mathsf{CR}^{\Gm^2}_{k[u]}$. 

\begin{remark}\label{remark: whatisGm2ku}
To stay consistent with Remark ~\ref{remark: gradings}, when we write $\mathsf{CR}^{\Gm^2}_{k[u]}$ we require that $k[u]$ is equipped with a $\Gm$-action such that $\op{deg}u = (-1,1)$. Thus objects $S$ in $\mathsf{CR}^{\Gm^2}_{k[u]}$ are equipped with a map $k[u]\to S$ which respects this grading, and likewise morphisms respect this structure. In other words, $\mathsf{CR}^{\Gm^2}_{k[u]}$ is opposite to the category of affine schemes equipped with $\Gm^2$-actions and an {\em equivariant} map to $\mathbb{A}^1_k$. \end{remark}

\begin{definition}\label{definition: diagonalfunctor}
 Let $\Delta: \mathsf{CR}_k^{\Gm} \to \mathsf{CR}^{\Gm^2}_{k[u]}$ denote the functor determined by 
 \begin{displaymath}
  \Delta(R) := R[u,u^{-1}],
 \end{displaymath} where $\Delta (R)$ is equipped with the $\Gm^2$-action from Lemma ~\ref{lemma: basic action}. 
\end{definition}

\noindent We may view $\Delta(R)$ with its $\Gm^2$-action and its two $R$-module structures via $\pi$ and $\sigma$ as an object of the derived category $\op{D}^b(\op{mod}^{\Gm^2}R \otimes_k R)$ $\mathbb{Z}^2$-graded modules by regarding $\Delta (R)$  as the complex concentrated in homological degree zero. The notation $\Delta (R)$ introduced above is justified by the following.

\begin{lemma} \label{lemma: kernel of the identity}
 The object $\Delta (R)\in \op{D}^b(\op{mod}^{\Gm^2}R \otimes_k R) $ is the Fourier-Mukai kernel of the identity functor on $\op{D}^b\op{mod}^{\Gm}(R)$. \end{lemma}
\begin{proof}  The Fourier-Mukai transform associated to $\Delta (R)$ is given by \begin{equation}\label{equation: diagonalFM} \Phi_{\Delta (R) }(M) := \pi_\ast (\Delta (R) \otimes^{\mathbb{L}}\sigma^\ast M)^{\Gm} \end{equation} where $M\in\op{D}^b\op{mod}^{\Gm}(R)$ and the functors $\pi_\ast$ and $\sigma^\ast$ are derived on the left and on the right respectively. In equation ~\eqref{equation: diagonalFM} we recall that taking a derived-push forward in the equivariant setting takes invariants, see e.g. \cite[Section 6]{BKR}, which is well-defined since $\Gm$ is reductive and so its functor of invariants is exact. By Remark \ref{remark: gradings}, the $\Gm$-invariants referred to in Equation ~\eqref{equation: diagonalFM} mean taking invariants on the left factor in the $\mathbb{Z}^2$-grading, i.e. taking terms of degree $(0,\ast )$. Since $\Delta (R)$ is flat (via either module structure), we then compute \[ (\Delta (R) \otimes^{\mathbb{L}}\sigma^\ast M)^{\Gm} = (\Delta (R) \otimes\sigma^\ast M)^{\Gm}  = (R[u,u^{-1}]\otimes\sigma^\ast M)_{(0,\ast )} \cong \sigma^\ast M\] so that $\Phi_{\Delta (R)}(M) = \pi_\ast\sigma^\ast M$. Likewise, for the inverse Fourier-Mukai transform,  we have \[ M\mapsto \sigma_\ast (\Delta (R) \otimes^{\mathbb{L}}\pi^\ast M)^{\Gm}\cong \sigma_\ast\pi^\ast M.\] To prove the lemma, we thus need to exhibit the two natural isomorphisms  \begin{align*}
  M & \cong \sigma_* \pi^* M \\
 M  & \cong \pi_* \sigma^* M.
 \end{align*}
for any $M\in \op{D}^b(\op{mod}^{\Gm}R)$. For $M\in \op{mod}^{\Gm}(R)$ these isomorphisms are respectively given by composing the maps  \begin{align*}
  M & \to \sigma_* \pi^* M \\
  m & \mapsto \sigma_M(m) \in M[u,u^{-1}] \cong \Delta (R)\otimes_R M
 \end{align*}
and
 \begin{align*}
  M & \to \pi_* \sigma^* M \\
  m & \mapsto m \in M[u,u^{-1}] \cong \Delta (R)\otimes_R M. 
 \end{align*} and where  ($\Delta (R)\otimes M)^{\Gm}\cong M$. The result follows for any $M\in \op{D}^b(\op{mod}^{\Gm}R)$ since $\Delta (R)$ is again flat via either module structure.

\end{proof}

The following definition, though simple, is crucial for this article.

\begin{definition}\label{definition: Q} 
 Given an object $R$ of $\mathsf{CR}_k^{\Gm}$, we define 
\begin{equation}
Q(R) := \langle \pi (R),\sigma (R), u\rangle\subseteq R[u,u^{-1}]\end{equation} to be the $k$-subalgebra  of $R[u,u^{-1}]$ generated by $u$ and the images of the co-action and projection maps. We regard $Q(R)$ as an object of $\mathsf{CR}^{\Gm^2}_{k[u]}$. 
\end{definition}

By construction, the maps $\pi$ and $\sigma$ both have images in $Q(R)$. We thus have maps 
\begin{equation}\label{equation: pands}
\begin{tikzcd}
R \ar[r, shift left, "p"] \ar[r, shift right, swap, "s"] & Q(R)
\end{tikzcd}
\end{equation}
which equal $\pi$ and $\sigma$ respectively (in general $Q(R)\neq R[u,u^{-1}]$, hence we reserve different symbols for these maps). The same construction as in Lemma ~\ref{lemma: basic action} then shows that $Q(R)$ has $\Gm^2$-action. In terms of the $\mathbb{Z}$-grading on $R$, we may equivalently write 
\begin{equation}\label{equation: generatorsforQ}
Q(R) = \langle \bigoplus_{i<0}R_iu^i , R[u]\rangle
\end{equation} where $R=\bigoplus_{i\in\mathbb{Z}}R_i$. Notice that when $r\in R$ is homogeneous with non-negative degree, $s(r) = ru^{\op{deg}(r)} \in R[u]$, and thus, in addition to $R[u]$, only the additional summands $R_iu^i$ with $i$ negative are required to generate. 

\begin{example}\label{example: affinespaceQ}
 Let $R$ be a polynomial ring equipped with a $\Gm$-action (i.e. a $\mathbb{Z}$-grading).  Relabelling variables as necessary, write \begin{equation}R = k[x_1^+,\ldots , x_k^+, x^-_1,\ldots , x^-_l]\end{equation} where  each $x_i^+$ has non-negative (possibly zero) degree and each $x_j^-$ has strictly negative degree. Write the degrees as  $a_i = \op{deg}(x_i^+)\leq 0$ and $b_j = \op{deg}(x^-_j)<0$.  As in Equation ~\eqref{equation: generatorsforQ}, $Q(R)$ is generated (over $k$) by \[ u, x_1^+,\ldots x_k^+, x_1^-,\ldots x_l^-,x_1^-u^{b_1},\ldots , x_l^-u^{b_l}.\] To compress notation, write $\mathbf{x}^+$ for the set of $x_i^+$'s and similarly for $\mathbf{x}^-$.  Setting $y^-_j = x_j^-u^{b_j}$ for $j=1,\ldots ,l$ and inspecting the relations, one has 
\[  Q(R)=k[u,\mathbf{x}^+, \mathbf{x}^-, y_{1}^-,\ldots ,y_l^-] /(y^-_{1}u^{-b_{1}}-x^-_{1} , \ldots , y_l^-u^{-b_l}-x^-_l ). 
\] It is convenient to write this more symmetrically by setting $y_i^+ = x_i^+u^{a_i}$ for $i=1,\ldots , k$ as well so that, in compressed notation,  \begin{align} Q(R) & = k[u,\mathbf{x}^+,\mathbf{x}^-,\mathbf{y}^+,\mathbf{y}^-]/(\mathbf{x}^+u^{\mathbf{a}} - \mathbf{y}^+, \mathbf{y}^-u^{-\mathbf{b}}-\mathbf{x}^-)\label{equation: Qforaffinespace} \\&= k[u,\mathbf{x}^+,\mathbf{y}^-].\end{align} In these variables, the maps $s,p : R\to Q(R)$ are given by:  
\begin{align}
 p(x_i^+)  &= x_i^+   &   p(x_j^-)  &= y_j^-u^{-b_j}\label{equation: affinespacemodule1}   \\
 s(x_i^+) &= x_i^+u^{a_i}       &   s(x_j^-) &= y_j^-.\label{equation: affinespacemodule2} 
\end{align}
\noindent We will interpret these maps geometrically in Example ~\ref{example: affinespaceQgeo}. Also, in Subsection ~\ref{section: flops}  we will revisit this example (when $k=l$) while studying the Atiyah flop. 

\end{example}
\begin{example}\label{example: nonnegative} Suppose that $R$ is any non-negatively graded ring, i.e. $\Gm$ acts on $R$ with non-negative weights. Then $Q(R)$ is actually generated by $\pi(R)$ and $u$ as a $k$-algebra. As a module over $R$ via $p$, we thus have $Q(R) \cong R[u]$ or, geometrically, $\op{Spec}Q(R)\cong \A^1_k\times \op{Spec}R$. 

If $R$ is instead non-positively graded, then of course $Q(R)$ is not generated by just $\pi (R)$ and $u$. However, we may still construct an isomorphism $Q(R)\cong R[u]$ where now $Q(R)$ has the $R$-module structure via $s$. Indeed, one always has the map $R[u]\to Q(R)$ given by $r\mapsto \sigma (r)$ and $u\mapsto u$. This map is always injective, and when the weights are all non-positive, it is easily checked to be surjective as well. 
\end{example}

For later use we record an elementary property that $Q$ enjoys with respect to polynomial extensions; in particular, the next lemma shows that $\op{Spec}Q(R[x])\cong\op{Spec}\mathbb{A}^1_k\times \op{Spec}Q(R)$. 

\begin{lemma}\label{lemma: Rx}Given $R \in \mathsf{CR}_k^{\Gm}$, endow $R[x]$ with a $\Gm$-action by giving $x$ degree $a$. Then \begin{equation}\label{equation: Qofpoly} Q(R)[y]\cong Q(R[x])\end{equation}  where $y$ has degree $(a,0)$ if $a \geq 0$ and degree $(0,a)$ if $a \leq 0$. 

\end{lemma} 

\begin{proof}Assume the degree of $x$ is $a \leq 0$. The case $a \geq 0$ is analogous. 
We have a map 
\begin{align*}
Q(R)[y] & \to Q(R[x]) \\
y & \mapsto u^a x
\end{align*}
This map is clearly surjective, and 
 has kernel if and only if $u^a x$ is algebraic over $Q(R)$. If so, then $u^a x$ is algebraic over $R[u,u^{-1}]$ in $R[x,u,u^{-1}]$ which is impossible, since $u$ is a unit. The statement about the weight of $y$ is just the weight of $u^ax$ via the grading from Remark ~\ref{remark: gradings}. 
\end{proof} 

\begin{remark} The reader desiring a more geometric understanding of $Q(R)$ will find such a discussion in Section ~\ref{section: compactifications}.
\end{remark} 

\begin{remark}\label{remark: drinfeld} Our definition of $Q(R)$ is actually equivalent to the affine case of a construction of Drinfeld, see \cite[Section 2.3]{Drinfeld}. Namely, if $R=\bigoplus R_i$ is $\mathbb{Z}$-graded, Drinfeld considers the the $k[t]$-algebra $\tilde{R}$ generated by all symbols $\tilde{r}$ where $r\in R$, and which are required to be $k$-linear and subject to the relations \[ \tilde{r_1}\tilde{r_2} = t^{\mu (n_1, n_1)}\tilde{r_1}\tilde{r_2} \] where each $r_i\in R_{n_i}$ and where \[ \mu (m, n) := \textrm{min}\{|m|,|n| \}\] if $m$ and $n$ have opposite signs, and $\mu (m,n)=0$ otherwise. It is easy to check that the assignment  $u\mapsto t$, $r_i\mapsto t^{n_i}r_i$ for $i<0$, and $r_i\mapsto r_i$ for $i\geq 0$ induces an isomorphism $Q(R) \cong \tilde{R}$.  In the Appendix we will discuss in more detail the role of our $Q(R)$ in the context of Drinfeld's article.  
\end{remark} 

The study of $Q(R)$ as a Fourier-Mukai kernel is, in a sense, the fundamental goal of this article. To this end, we first note that formation of $Q(R)$ is functorial. 

\begin{lemma}\label{lemma: Q definition}  The assignment of $Q(R)$ to $R$ defines a functor
 \begin{displaymath}
  Q: \mathsf{CR}_k^{\Gm} \to \mathsf{CR}^{\Gm^2}_{k[u]}
 \end{displaymath}
 which preserves surjective morphisms.  
\end{lemma}

\begin{proof}
Given $\phi: S\to R$, the corresponding map $S[u,u^{-1}]\to R[u,u^{-1}]$ is surjective when $\phi$ is.  Furthermore, one checks that the image of $Q(S)$ under this map is $Q(R)$. 
\end{proof} 

We can relate the functors $Q$ and $\Delta$ by a natural transformation by considering the inclusions 
\begin{equation}\label{equation: eta} 
\eta_R: Q(R) \hookrightarrow R[u,u^{-1}] = \Delta (R)
\end{equation}
which come from the definition of $Q(R)$ as a subalgebra of $\Delta (R)$. These inclusions behave predictably; in particular, we have the following commutative diagrams:
 \begin{equation}
\begin{tikzcd}
Q(R) \ar[r, "\eta_R"] &  R[u, u^{-1}]  \\
& R \ar[u, "\pi"]  \ar[ul, "p"]
\end{tikzcd} 
\tand 
\begin{tikzcd}
Q(R) \ar[r, "\eta_R"] &  R[u, u^{-1}]  \\
& R. \ar[u, "\sigma"]  \ar[ul, "s"]
\end{tikzcd} 
\label{equation: Q(R)commutes}
 \end{equation}

\begin{definition} Let $\eta : Q \to \Delta$ be the natural transformation of functors induced by the inclusions $\eta_R : Q(R) \hookrightarrow \Delta (R)$. Here $Q$ and $\Delta$ are the functors $\mathsf{CR}_k^{\Gm} \to \mathsf{CR}^{\Gm^2}_{k[u]}$ from Lemma  ~\ref{lemma: Q definition} and Definition ~\ref{definition: diagonalfunctor} respectively. 

\end{definition} 

\noindent We will sometimes abuse notation and denote a map $\eta_R$ by $\eta$ when it is clear from context that we are working with rings and not the natural transformation. 

 Another way of understanding the relation between $Q$ and $\Delta$ is the following result, which shows in particular that they become identified after localizing by elements of non-zero degree. 

\begin{lemma} \label{lemma: Q is trivial on semi-stable}  Let $R$ be an object of $\mathsf{CR}_k^{\Gm}$. View $Q(R)$ as a right $R$-module via $s$ and a left $R$-module via $p$, and view $\Delta (R)$ as a right $R$-module via $\sigma$ and a left $R$-module via $\pi$. Let  $r \in R$ be a homogeneous element, and let $R\rightarrow R_r$ be the corresponding homogeneous localization. If $\op{deg}(r)> 0$, then
 \begin{displaymath}
 1 \otimes_s \eta : R_r \otimes_s Q(R) \to R_r \otimes_\sigma \Delta(R)
 \end{displaymath}
 is an isomorphism. 
If $\op{deg}(r) <0$, then 
 \begin{displaymath}
  \eta \mathbin{_p\otimes} 1: Q(R) \mathbin{_p\otimes} R_r \to \Delta(R) \mathbin{_\pi\otimes} R_r
 \end{displaymath}
 is an isomorphism. 
If $\op{deg}(r) = 0$, then \[  Q(f) \mathbin{_p\otimes}1 :Q(R)\mathbin{_p\otimes} {R_r} \to Q(R_r)\otimes R_r\cong Q(R_r) \] is an isomorphism. 
\end{lemma}

\begin{proof} In all three cases, injectivity of the maps holds because localization is flat and $\eta$ and $Q(f)$ are inclusions. The only non-trivial part of verifying surjectivity in the cases of non-zero degree is to demonstrate that $u^{-1}$ lies in the image. Indeed, in the case where positive degree one has that $(1\otimes_s \eta )(\frac{1}{r}\otimes u^{\op{deg}(r)-1}r) = u^{-1}$ and in the case of negative degree one has that $ ( \eta \mathbin{_p\otimes} 1)(s(r)u^{-1 -\op{deg}(r)]} \otimes \frac{1}{r}) = u^{-1}$. Surjectivity in the degree zero case is clear.  
\end{proof}

We conclude this section with some natural loci coming from a $\Gm$-action. 

\begin{definition}\label{definition: plusminus} If $R$ is an object of $\mathsf{CR}_k^{\Gm}$ let $I^{\pm}\subseteq R$  be the ideals generated by elements of positive or respectively negative degree. Set $R^+ := R/I^-$, and $R^- := R/I^+$ (note the swap in signs), and set $R^0 := R/(I^+,I^-)$. 
\end{definition} 

\noindent The reason for defining $R^0$ and $R^{\pm}$ in the above fashion is to match notation with Definition ~\ref{definition: attractingrepelling} below, which gives their geometric description.  

\begin{definition}\label{definition: attractingrepelling}  Let $X$ be any $k$-scheme equipped with a $\Gm$-action. Equip $\op{Spec}(k)$ with the trivial $\Gm$-action. Let $\mathbb{A}^1_+$ denote $\mathbb{A}^1$ equipped with its usual $\Gm$-action by scaling, and let $\mathbb{A}^1_-$ denote $\mathbb{A}^1$ equipped with the inverse action $t\cdot x = t^{-1}x$. Then 
\[ X^0 := \op{Hom}^{\Gm}(\op{Spec}(k),X) = \{x\in X\; |\;  t\cdot x = x\;\op{for}\; \op{all} \;t\in\mathbb{A}^1 \} \] is the \newterm{fixed point locus} for the $\Gm$-action, \[ X^+ := \op{Hom}^{\Gm}(\mathbb{A}^1_+, X) = \{x\in X\; |\; \lim_{t\rightarrow\infty} t\cdot x\;\op{exists}\;\op{in}\; X \}\] is the \newterm{attracting locus},  and \[ X^- :=\op{Hom}^{\Gm}(\mathbb{A}^1_-, X) = \{x\in X\; |\; \lim_{t\rightarrow\infty} t^{-1}\cdot x\;\op{exists}\;\op{in}\; X \} \] is the \newterm{repelling locus}. 
\end{definition} 

\noindent We refer to \cite[Section 1]{Drinfeld} for basic properties of these loci. As alluded to above, when $X := \op{Spec}R$ is affine and equipped with a $\Gm$-action, we have \[ X^0 = \op{Spec}R^0 \]  and \[ X^{\pm} = \op{Spec}R^\pm = V(I^\pm ) \]  using the notation of Definition ~\ref{definition: plusminus}. By construction there are inclusions $R^0\subset R^{\pm}$ which give maps \begin{equation}\label{equation: blades} q^{\pm} : X^{\pm}\to X^0.
\end{equation} In terms of Definition ~\ref{definition: attractingrepelling} the maps $q^{\pm}$ can equivalently be understood as being induced by the $\mathbb{G}_m$-equivariant inclusions $\op{Spec}(k) \hookrightarrow \mathbb{A}^1_{\pm}$. Geometrically, this means that the maps $q^{\pm}$ correspond to taking a point to its limit along either $\Gm$-action or the inverse $\Gm$-action. 

\begin{remark}\label{proposition: plusminusfromQ} 
The ideals $I^{\pm}$ can be recovered directly from $Q(R)$ equipped with its maps $p$ and $s$.  Namely, 
 \begin{align*}
  I^+ & = s^{-1}(uQ(R)) \subseteq R  \\
  I^- & = p^{-1}(uQ(R)) \subseteq R. 
 \end{align*}
\end{remark} 

\subsection{Flop equivalences via Q(R)}\label{section: flops}

We now observe that $Q(R)$ provides the derived equivalence constructed by Bondal and Orlov for the standard Atiyah flop \cite[Section 3]{BO}. For $n\geq 2$, let \begin{equation}\label{equation: atiyahring} R = k[x_1^+,\ldots,x_n^+,x_1^-,\ldots,x_n^-]\end{equation} with the $\Gm$-action given by taking $\op{deg}(x_i^+) = 1$ and $\op{deg}(x_i^-) = -1$ for each $i$. The ideals $I^{\pm}\subset R$ from Definition ~\ref{definition: plusminus} are $I^+ = (x_1^+,\ldots,x_n^+) = (\mathbf{x}^+)$ and $I^- = (x^-_1,\ldots,x^-_n) = (\mathbf{x}^-)$.  Let $X=\op{Spec}R = \mathbb{A}^{2n}$ and set 
\begin{align*}
 U^{\pm} & := X \setminus V(I^{\pm}) \\
 X/ \! /{\pm} & := U^{\pm}/\Gm\\
 X/ \! /0 & := \op{Spec} R^{\Gm}. 
\end{align*}

\begin{remark} The invariant subring $ R^{\Gm}$ should not be confused with $R^0$ from Definition ~\ref{definition: plusminus} where the defining ideal is $(I^+,I^-)$; in other words, the invariant theory quotient $X/ \! /0$ is not the same thing as the fixed locus $X^0$. 

\end{remark} 

\noindent For the standard Atiyah flop, $X/ \! /_{\pm}\Gm$ are both isomorphic to the total space of $\mathcal{O}(-1)^{\oplus n}$ on $\mathbb{P}^{n-1}$ and $X/ \! /_0\Gm$ is a singular affine quadric. Let \[ p^{\pm} : X/ \! /+\longrightarrow X/ \! /0\]  be the corresponding birational contractions.

 The diagram

\begin{center}
\begin{tikzpicture}[scale=1,level/.style={->,>=stealth,thick}]
	\node (a) at (-2,2) {$X/ \! /+$};
	\node (b) at (2,2) {$X/ \! /-$};
	\node (c) at (0,0) {$X/ \! /0$};
	\draw[level,dashed] (a) -- (b) ;
	\draw[level] (a) -- (c) ;
	\draw[level] (b) -- (c) ;
\end{tikzpicture}
\end{center}
is the prototypical example of a flop. \cite[Theorem 3.9]{BO} proves that the functor \begin{equation}\label{equation: atiyahequivalence}
p^-_* p^{+\ast} : \op{D}^b(\op{coh}X/ \! /+) \to \op{D}^b(\op{coh}X/ \! /-). 
\end{equation}
is an equivalence of derived categories of coherent sheaves (where the functors $p^-_* $ and $p^{+\ast}$ are, of course, derived on the right and left, respectively). For this example, the actions of $\Gm$ on $U^\pm$ are free, and so we may identity
\begin{equation}\label{equation: atiyahequivariant}
\op{D}^b(\op{coh}X/ \! /+) = \op{D}^b(\op{coh}^{\Gm}U^\pm )
\end{equation}
and recognize the Bondal-Orlov flop equivalence as an equivalence of equivariant derived categories. 

\begin{remark} Taking the functor to be $p^-_* p^{+\ast} $ is equivalent to taking the fiber product along $p^{\pm}$ as the Fourier-Mukai kernel of the functor.  Instead of the fiber product, one could also take as a Fourier-Mukai kernel a common blow-up resolving the birational map. For the Atiyah flop, though, the fiber product and the resolution are isomorphic, see \cite[Proposition 5.5]{Kaw}. 
\end{remark} 

We now observe that the fiber product agrees with an appropriate restriction of $Q(R)$.

\begin{proposition} \label{proposition: BO = Q}
 With $Q(R)$ as in Equation ~\eqref{equation: atiyahring}, let \begin{equation}\label{equation: Qwc} 
 Y^{wc} :=  \op{Spec} Q(R) \times_{X \times X} (U^+ \times U^-) 
\end{equation} 
 denote the restriction of $Q(R)$ to the open subset $U^-\times U^+\subset X\times X$. (The ``wc" stands for ``wall-crossing".) Then there is a natural isomorphism
 \begin{equation}\label{equation: atiyahwc}
  Y^{wc} \cong U^+ \times_{X/ \! /0} U^-.
 \end{equation} In other words, if we regard $Q^{wc} = \Gamma(Y^{wc},\mathcal O)$ as an object of $\op{D}^b(\op{coh}^{\Gm^2} U^-\times U^+)$, then the Fourier-Mukai functor \[ \Phi_{Q^{wc}} : \op{D}^b(\op{coh}^{\Gm}U^-) \to \op{D}^b(\op{coh}^{\Gm}U^+) \] is an equivalence and $\Phi_{Q^{wc}}$ is naturally isomorphic to to $p^-_* p^{+\ast} $ under the identification in Equation ~\eqref{equation: atiyahequivariant}. 
\end{proposition}
\begin{proof}
The universal property of the fiber product gives a map \begin{equation}\label{equation: fiberQringmap}\phi: R \otimes_{R^{\Gm}} R \to Q(R)\end{equation} which is, explicitly, given by $r_1\otimes r_2\mapsto r_1s(r_2)$, i.e. $\phi = p \otimes_{R^{\Gm}} s$. Taking $k=n$ and $l=n$ in Example ~\ref {example: affinespaceQ} to describe $Q(R)$, we have: \begin{displaymath}
 \phi :  \frac{k[\mathbf{x}^+,\mathbf{x}^-,\mathbf{y}^+,\mathbf{y}^-] }{(\mathbf{x}^+\mathbf{y}^- - \mathbf{y}^+\mathbf{x}^-)} \to k[u, \mathbf{x}^+, \mathbf{y}^-, u^{-1}y_1^-,\ldots u^{-1}y_n^-]. \end{displaymath}
 We can work on a cover of $U^+ \times U^-$ given by inverting the monomials in $\mathbf{x}_1$ and $\mathbf{y}_2$. One can check that inverting any such pair reduces $\phi$ to an isomorphism. The statement about $\Phi_{Q^{wc}}$ being an equivalence is then just a rephrasing of the Bondal-Orlov equivalence, i.e. Equation ~\eqref{equation: atiyahequivalence}. 
\end{proof}

\begin{example}\label{example: mukai} Let \[ R = k[x_1,\ldots ,x_n,y_1,\ldots , y_n] /(x_1y_1+\ldots +x_ny_n)\] with $n\geq 2$ and with each $\op{deg}(x_ i )=1$ and each $\op{deg}(y_j) =-1$. Here $\op{Spec}R/ \! /+\dashrightarrow\op{Spec}R/ \! /-$ is the elementary Mukai flop. Note that here $\op{Spec}R$ is singular. In Example \ref{example: Q2=Q fails} we will see that such singularities lead to a poorly behaved Fourier-Mukai functor associated to $Q(R)$, and we will correct this behavior in Example ~\ref{example: derivedexample} by replacing $Q(R)$ with a suitable affine derived scheme $Q_{\op{der}}(R)$. It is not difficult to show that a suitable generalization of Proposition ~\ref{proposition: BO = Q} holds once this correction is made. 
\end{example} 

\section{Functors from compactifications}\label{section: compactifications}

\subsection{Partial compactifications of group actions}\label{section: compactificationsfirst}

We now address the geometric interpretation of $Q(R)$. In this section,  in order to provide context for the connection between $Q(R)$ and geometric invariant theory (GIT), we will give definitions and results for  varieties which are not necessarily affine and actions by algebraic groups other than $\Gm$. The reader only concerned with the level of generality required for the later portions of this paper may well assume that $G=\Gm$ and $X=\op{Spec}R$ is affine throughout this section as well; under these assumptions Proposition ~\ref{proposition: Qisacompactification} and Proposition ~\ref{proposition: boundaryofQ} summarize the relevant statements needed from this subsection. 

 Let $G$ be an algebraic group acting on a variety $X$. Let $\widehat{\sigma }: G\times X\to X$ be the action map and $\widehat{\pi}: G\times X\rightarrow X$ the projection. The space $G\times X$ itself admits a $G\times G$ action given by \begin{equation}\label{equation: GxXaction}
  (g_1,g_2) \cdot (g,x) = (g_2 g g_1^{-1}, \widehat{\sigma }(g_1 ,x))
  \end{equation}
  which makes $\widehat{\pi}$ and $\widehat{\sigma}$ equivariant. With this action, the map $G\times X \to X\times X$ given by $(g,x)\mapsto (x,\widehat{\sigma} (g,x)) $ becomes equivariant, where $X\times X$ has the obvious $G\times G$ action. 

\begin{definition}\label{definition: compactification} Let $\tilde{X}$ be an algebraic variety together with an action of $G \times G$ which is equipped with a $G \times G$-equivariant open immersion 
\[
i: G \times X \hookrightarrow \tilde{X},
\]
and a $G \times G$-equivariant morphism
\[
(\widehat p, \widehat s) : \tilde X \to X \times X
\]
such that the following diagram commutes:
\begin{center}\label{fig: compactification}
\begin{tikzpicture}
\node (a) at (-1,-1) {$G\times X$};
\node (b) at (1,-1) {$X$};
\node (c) at (1,1) {$\tilde{X}$};
\path[->,font=\scriptsize,>=angle 90]
(a) edge  node[above] {$i$} (c)
([yshift= 2pt]a.east) edge node[above] {$\widehat{\sigma}$} ([yshift= 2pt]b.west)
([yshift= -2pt]a.east) edge node[below] {$\widehat{\pi}$} ([yshift= -2pt]b.west)
([xshift= -2pt]c.south) edge node[left] {$\widehat{p}$} ([xshift= -2pt]b.north)
([xshift= 2pt]c.south) edge node[right] {$\widehat{s}$} ([xshift= 2pt]b.north);
\end{tikzpicture}
\end{center}
i.e. $\widehat{p}\circ i = \widehat{\pi}$ and  $\widehat{s}\circ i= \widehat{\sigma}$. 
In the above situation, we say that  $\tilde{X}$ equipped with the maps $ \widehat{p}, \widehat{s}, i$, is a \newterm{partial compactification of the action} of $G$ on $X$. \end{definition}

The following result shows that $Q(R)$ encodes exactly this kind of structure when $G=\Gm$ and $X$ is affine. 

\begin{proposition}\label{proposition: Qisacompactification} Let $\op{Spec} R$ be an affine variety with a  $\Gm$-action.  Consider the maps
\[
\begin{tikzcd}
\op{Spec} Q(R) \ar[r, shift left, "\widehat{p}"] \ar[r, shift right, swap, "\widehat{s}"] & \op{Spec} R.
\end{tikzcd}
\]
corresponding to the ring maps $p,s : R\to Q(R)$ and let
\[
\widehat{\eta} : \op{Spec} Q(R) \to \op{Spec} R[u, u^{-1}] = \Gm \times \op{Spec} R
\]
 be the open immersion corresponding to $\eta : Q(R) \to R[u, u^{-1}]$ . Then $\op{Spec} Q(R)$ together with the maps $\widehat{p}, \widehat{s}, \widehat{\eta}$ form a partial compactification of the action of $\Gm$ on $\op{Spec} R$. 

\end{proposition}

\begin{proof} This follows easily from previous observations: the equivariant structure on $Q(R)$ is determined by the $\mathbb{Z}^2$-grading given in Remark ~\ref{remark: gradings}.  The fact that $\widehat{\eta}$ is an open immersion follows from the fact that it is the localization along $u$.  That Figure ~\ref{fig: compactification} commutes is exactly dual to the commutativity of the diagrams in \eqref{equation: Q(R)commutes}.
\end{proof}

\begin{example}\label{example: affinespaceQgeo} Let us revisit Example ~\ref{example: affinespaceQ} where we considered the case of a $\Gm$-action on a polynomial ring, i.e. we now consider $\op{Spec}(R) = \mathbb{A}^{k+l}$ where the first $k$-coordinates are acted on with non-negative weights and the last $l$-coordinates with strictly negative weights. Following Example ~\ref{example: affinespaceQ}, write \[\mathbb{A}^{k+l+1}  = \op{Spec}R[u,\mathbf{x}^+,\mathbf{y}^-] =  \op{Spec}Q(R). \]  The $p$ and $s$ module structures computed in Equations ~\eqref{equation: affinespacemodule1} and ~\eqref{equation: affinespacemodule2} determine the maps $\hat{p}$ and $\hat{s}$ in the commutative diagram 
\begin{center}
\begin{tikzpicture}
\node (a) at (-1,-1) {$\gm \times \mathbb A^{k+l}$};
\node (b) at (1,-1) {$\mathbb A^{k+l}$};
\node (c) at (1,1) {$\mathbb A^{k+l+1}$};
\path[->,font=\scriptsize,>=angle 90]
(a) edge  node[above] {$i$} (c)
([yshift= 2pt]a.east) edge node[above] {$\hat\sigma$} ([yshift= 2pt]b.west)
([yshift= -2pt]a.east) edge node[below] {$\hat\pi$} ([yshift= -2pt]b.west)
([xshift= -2pt]c.south) edge node[left] {$\hat p$} ([xshift= -2pt]b.north)
([xshift= 2pt]c.south) edge node[right] {$\hat s$} ([xshift= 2pt]b.north);
\end{tikzpicture}
\end{center}
where $\hat{\sigma}$ and $\hat{s}$ are the action and projection maps, and $i$ is an inclusion chosen to make the diagram commute. Explicitly: 
\begin{align*} 
i(u,x_1^+, \ldots ,x_k^+,x^-_1,\ldots ,x^-_l) &=  (u,x_1^+,\ldots , x_k^+, u^{b_1}x^-_1,\ldots ,u^{b_l}x^-_l) \\ 
\hat{p}(u, x_1^+,\ldots ,x_k^+ ,y^-_1 , \ldots ,y^-_l) &=  (x_1^+,\ldots x_k^+, u^{-b_1}y^-_1,\ldots ,u^{-b_l}y^-_l ) \\
\hat{s}(u, x_1^+,\ldots ,x_k^+ ,y^-_1 , \ldots ,y^-_l) &=  (u^{a_1}x_1^+,\ldots u^{a_k}x_k^+, y^-_1, \ldots ,y^-_l ). 
\end{align*}
Note that the map $i$ is not the ``naive" inclusion in these coordinates.\end{example}

The data of the partial compactification of an action of $G$ on $X$ automatically encodes a notion of a boundary on $\tilde{X}$ as well as distinguished ``unstable" and ``semistable" loci in $X$, as in the following definition. 

\begin{definition}\label{definition: boundary}
If $\tilde{X}$ is a partial compactification of an action $\sigma : G\times X\to X$, we define the boundary of $\tilde{X}$ to be 
\[
\partial_{\tilde{X}} := \tilde{X} \backslash i(G \times X), 
\]
 the $\widehat{s}$-\newterm{unstable locus}  to be  \[ X^{\op{us}} =\widehat{s}(\partial_{\tilde{X}}),\] and the $\widehat{s}$-\newterm{semistable locus} to be  \[ X^{\op{ss}}=X \setminus X^{\op{us}}.\] 
\end{definition}

Note that $X^{ss}$ itself admits a $G$-action as the $G\times G$ action on $\tilde{X}$ extends the action on $G\times X$ and $\widehat{s}$ was assumed equivariant. In this article we are primarily concerned with the affine group scheme $\Gm$ and the case where $X = \op{Spec}R$ is affine, in which case $\partial_{\tilde{X}}$ is a automatically a divisor in $\tilde{X}$ (although $X^{us} = \widehat{s}(\partial_{\tilde{X}})$ will rarely be a divisor in $X$).

\begin{remark} \label{remark: sameasGITss} 
 The terminology in Definition ~\ref{definition: boundary} is chosen to emphasize the connection with geometric invariant theory (GIT).  The next lemma demonstrates the precise relationship between our notion of $\hat{s}$-semistability and semistability in GIT. We will also explain the relationship for actions by higher rank tori in Example ~\ref{example: toric} and Proposition ~\ref{proposition: toricsemistable}.
\end{remark}

\begin{proposition}\label{proposition: boundaryofQ} Let $\Gm$ act on $X = \op{Spec} R$ and let $\op{Spec} Q(R)$ partially compactify the action as in Proposition ~\ref{proposition: Qisacompactification}. Then the $\widehat{s}$-unstable locus is the repelling locus $X^-$ as defined in Definition ~\ref{definition: attractingrepelling}. In particular, $X^{ss} = X\setminus X^-$.
\end{proposition}

\begin{proof}  Since $Q(R) =  \langle u, \pi (R), \sigma (R)\rangle$, the ideal which defines the boundary $\partial \subseteq \op{Spec }Q(R)$ is simply $\langle u\rangle \subseteq Q(R)$.  Hence, $\widehat{s}(\partial)$ is given by the ideal $\langle u\rangle \cap s(R)$. By the explicit gradings in Remark ~\ref{remark: gradings}, it is easy to see that $\langle u\rangle \cap s(R)=I^+$ is the ideal generated by all homogeneous elements of positive degree, which in turn is the ideal defining $X^-$. 
\end{proof} 

\noindent It follows immediately from the above proposition that, in the case of a $\Gm$-action on $X = \op{Spec}R$ equipped with the partial compactification of the action given by $\op{Spec}Q(R)$, one has 
\begin{equation}\label{equation: affinesemistable} X^{ss} = X\setminus X^- = \op{Spec}R \setminus \op{Spec}(R/I^+) = \bigcup_{r\in I^+}\op{Spec}R_r. 
\end{equation} 

We conclude this subsection with two brief but instructive examples of partial compactifications of actions which go beyond the case of a $\Gm$-action (and thus are not logically necessary for the remainder of this article). In the first example, $G$ is a higher rank torus acting on any affine space.  In the second example, $G$ is non-abelian.  

\begin{example}\label{example: toric} Let $T=\mathbb{G}_m^n$ be a torus and let $\chi_\ast (T) = \textrm{Hom}(T,\mathbb{G}_m)$  denote its character lattice. Suppose that $T$ acts on a ring $R$ with co-action map 
\begin{displaymath}
\sigma :R \to R\otimes_k k[T] \cong  R\otimes_k k[\chi_\ast (T)]\cong R[x_1^{\pm},\ldots , x_n^{\pm}].
\end{displaymath} Now, let $C\subset \chi_\ast (T)$ be any finitely generated submonoid. Let 
\begin{displaymath}
Q_T^C(R) = \langle R[C],\sigma (R)\rangle\subseteq R[\chi_\ast (T)] 
\end{displaymath} 
be the subalgebra generated by the image of the action and the monoid ring. A trivial generalization of Proposition ~\ref{proposition: Qisacompactification} shows that $\op{Spec} Q_T^C(R)$ is a partial compactification of the action of $T$ on $X = \op{Spec} R$. 

Now, suppose $\op{Pic}(\op{Spec}R)\otimes\mathbb{Q}$ is trivial. Then, following \cite[Section 2]{Thaddeus} or \cite[Section 3]{DH}, the cone of ample $T$-linearized $\mathbb{Q}$-divisors on $X=\op{Spec}R$ is a cone in $\chi_\ast (T)\otimes\mathbb{Q}$. This cone admits a chamber decomposition such that the relative interiors of the respective chambers correspond exactly to GIT quotients. That is, two characters $\chi$ and $\chi '$ lie in the relative interior of the same chamber exactly when the GIT semistable loci $X^{ss}(\chi )$ and $X^{ss}(\chi ')$ are equal

Of particular interest is the case where the monoid $C$ is itself the (integral points of) of a GIT chamber. Although not required for the remainder of the article, it seems worthwhile to observe that a generalization of Proposition ~\ref{proposition: boundaryofQ} holds in this setting which relates the notion of semistability given in Definition ~\ref{definition: boundary} with the usual notion of semistability in Geometric Invariant Theory (GIT). 

\begin{proposition}\label{proposition: toricsemistable} As above, let a torus $T$ act on a ring $R$ and suppose $\op{Pic}(\op{Spec} R)\otimes\mathbb{Q} =0$. Let $C$ be the monoid of integral points of the closure of a GIT chamber,  and let $L$ be any line bundle which lies in the relative interior of the same GIT chamber.  

Then the $\hat{s}$-semistable locus in the sense of Definition ~\ref{definition: boundary} equals the semistable locus for the action of $T$ on $X$ equipped with the linearization $L$ in the sense of GIT. 
\end{proposition}  

\begin{proof} The proof is essentially the same as that of Proposition ~\ref{proposition: boundaryofQ}. Indeed, the ideal defining the boundary of $\op{Spec}Q_T^C(R)$ is $\langle \chi_1\otimes 1,\ldots , \chi_k\otimes 1\rangle$ where the $\chi_i$ are any set of characters minimally generating the monoid $C$. Thus, ideal defining the $\hat{s}$-unstable locus in $X$ is $I:=\langle \chi_1\otimes 1,\ldots , \chi_k\otimes 1\rangle\cap s(R)$. However, localizing $R$ at $I$ is the same as localizing $R$ at $ \chi\otimes 1$ where $\chi\in C$ is any character in the relative interior of $C$, since we have $\chi =\sum a_i \chi_i$ with each $a_i>0$. 
\end{proof} 

\begin{remark}\label{remark: toric} 
As the reader may have already guessed, Example~\ref{example: toric} in particular shows how to realize toric varieties via partial compactifications of torus actions on affine space. 

Namely, if one specializes to the case where $R =k[x_1, \ldots , x_n]$ and so $\op{Spec}R=\mathbb{A}^n$, the above discussion reduces to torus actions on affine space. Here the GIT chambers agree with the cones in the GKZ-fan (see e.g. \cite[Section 14.4]{CLS}), and the GIT quotients are the realization of toric varieties via the Cox construction. It is not difficult to give an explicit combinatorial description of a rational polyhedral cone describing the affine toric variety $Q_T^C$ where $C$ is any chamber of the GKZ-fan, although doing so would be too much of a digression, so we leave this as an exercise to the combinatorially minded reader. 
\end{remark} 

\end{example} 

\begin{example}\label{example: grassmannian} 
Let $W$ be a vector space of dimension $k$ and $V$ be a vector space of dimension $n$ with $k \leq n$. Set $G = \op{Gl}(W)$ and $X = \op{Hom}(W, V)$ so that $G$ acts on $X$ by right multiplication. Then $\tilde{X} =\op{End}(W) \times X$ is a partial compactification of $G\times X$. The projection and multiplication maps extend naturally to maps $\tilde{X}\rightarrow X$, and here $X^{ss}$ is identified with the set of matrices of full rank, so that $X^{ss}/G$ is the Grassmannian $\textrm{Gr}(n,k)$.
\end{example}

\subsection{Functors from partial compactifications.} We now show how to associate functors between derived categories given the data of a partial compactification of an action. This is essentially done by taking the partial compactification itself as a Fourier-Mukai kernel. 

\begin{definition}\label{definition: Qingeneral} Given $\tilde{X}$ a partial compactification of a $G$-action on $X$ in the sense of Definition ~\ref{definition: compactification}, define 
\begin{equation}\label{equation: geometricQ} Q_{X,G} := (\widehat{p}\times\widehat{s})_\ast \mathcal{O}_{\tilde{X}}\in \mathrm{D}^{b}(\mathrm{Qcoh}^{G\times G} X\times X),
\end{equation} which is simply the derived push-forward of the structure sheaf under the extended action and projection maps.\end{definition} 

\begin{remark}\label{remark: Qxgwhenaffinegm} If $X$ is affine and $G=\Gm$ then, since the functor \newline$(\widehat{p}\times\widehat{s})_\ast $ is exact, Proposition ~\ref{proposition: Qisacompactification} shows that $Q(R)$ with its $p$-$s$-bimodule structure corresponds to the equivariant sheaf $Q_{\op{Spec}R, \Gm}$ . \end{remark} 

However, as suggested by the examples of flops from Section ~\ref{section: flops}, we are not literally interested in endofunctors $\mathrm{D}(\mathrm{Qcoh}^{G} X)\to\mathrm{D}(\mathrm{Qcoh}^{G} X)$ as would obtained by taking $Q_{X,G}$ as a kernel object. Rather, we are interested in functors $\mathrm{D}(\mathrm{Qcoh}^{G} X^{ss})\to\mathrm{D}(\mathrm{Qcoh}^{G} X)$ where $X^{ss}$ is the semistable locus from Definition ~\ref{definition: boundary}.

\begin{definition}\label{definition: Qss} Let $\tilde{X}$ be a partial compactification of a $G$-action on $X$. Then $Q^{ss}_{X,G}$ denotes the quasi-coherent sheaf on $X^{ss}\times X$ obtained by restricting $Q_{X,G}$ from $X\times X$. That is, \begin{equation}\label{equation: Qplusaspullback} Q^{ss}_{X,G} = (j\times\textrm{Id})^\ast Q \end{equation} where $j:X^{ss}\to X$ is the inclusion. 
\end{definition} 

\noindent For the purposes of the paper, we are primarily concerned with the case of a $\Gm$-action on an affine variety, for which we reserve different notation as follows. 

\begin{definition}\label{definition: Qplus} If $X=\op{Spec}R$ has a $\Gm$-action, then  $Q_+(R)$ denotes the quasi-coherent sheaf obtained by restricting the sheaf associated to $Q(R)$ to the quasi-affine variety $X^{ss}\times X = X\setminus X^- \times X$. \end{definition} 

\noindent We will frequently abuse notation and drop the implicit reference to $R$ and just write $Q_+$, and will also use \begin{equation}\label{equation: Qplus} Q_+ \in \op{D}^b(\op{Qcoh}^{\Gm\times\Gm} X^{ss}\times X)
\end{equation} to denote the corresponding object of the derived category. Taking $Q_+$ as a Fourier-Mukai kernel, we have the functor  \begin{equation}\label{equation: Qplusfunctor} \Phi_{Q_+} : \op{D}(\op{Qcoh}^{\Gm} X^{ss})\to \op{D}(\op{Qcoh}^{\Gm} X) .\end{equation}

\begin{remark}\label{remark: cpropremark} 
 For an arbitrary ring $R$, $Q$ is a module, hence a bounded complex. Since it may not be perfect in general, tensor product with $Q$ may not preserve boundedness unless $R$ has finite Tor dimension. When $X=\op{Spec}R$ is smooth, we will establish a ``cohomological properness'' result in Proposition ~\ref{proposition: window affine space case} which in particular implies that the essential image lands in the derived category of complexes with bounded and coherent coherent sheaves. We leave cohomological properness of $\Phi_Q$ beyond the smooth case for a fuller discussion.
\end{remark} 

 We now show that this functor is automatically faithful. 

\begin{proposition}\label{proposition: faithful} Let $\Gm$ act on $X=\op{Spec}R$. Let $Q_+$ be as in Equation ~\eqref{equation: Qplus}. Then the Fourier-Mukai functor $\Phi_{Q_+}$ from Equation ~\eqref{equation: Qplusfunctor} is faithful. 

\end{proposition} 
\begin{proof} This essentially follows from Lemma ~\ref{lemma: Q is trivial on semi-stable}. In more detail, by the open affine cover of $X^{ss}$ given in Equation ~\eqref{equation: affinesemistable}, one obtains the obvious cover of $X^{ss}\times X$. Lemma ~\ref{lemma: Q is trivial on semi-stable} says exactly that $Q$ restricts to $\Delta$ on each open affine subset of this cover, which by Lemma ~\ref{lemma: kernel of the identity} is the Fourier-Mukai kernel of the identity functor. If \begin{equation}\label{equation: semistableinclusion} j: X^{ss} \to X\end{equation} denotes the inclusion and $j^\ast: \op{D}(\op{Qcoh} X)\to \op{D}(\op{Qcoh} X^{ss})$ the restriction functor, we thus have that $j^\ast \circ \Phi_{Q_+} = \op{Id}$, and the result follows. 
\end{proof} 

\begin{remark}\label{remark: faithfulmoregeneral} It is not much more difficult to prove a strengthened version of Proposition ~\ref{proposition: faithful} valid for any kernel $Q^{ss}_{X,G}$ from Definition ~\ref{definition: Qss}. Indeed, the other main result from this Section, Proposition ~\ref{proposition: now we see some windows}, also admits such a generalization as well. Since we do not require such generality for the remainder of the article, though, we omit these generalizations for the sake of brevity. 
\end{remark} 

Fullness of the functor $\Phi_{Q_+}$ is more subtle and addressing this is, in a sense, the main technical content of Section ~\ref{section: bousfield} (when $\op{Spec}R$ is smooth) and Section ~\ref{section: Deriving Q} (in general). We will see in Lemma ~\ref{lemma: Q S bousfield triangle} that fullness of the functor $\Phi_{Q_+}$ is  intimately related to properties of the tensor product $Q(R) \mathbin{_s\otimes_p} Q(R)$, which we now study.

The tensor product $Q(R) \mathbin{_s\otimes_p} Q(R)$ inherits a $\Gm^3$-action from the $\Gm^4$-action on $Q(R)\otimes_k Q(R)$. Explicitly, let $(a, b)\in\Z^2$ denote the weight of a homogenous element of $r\in Q(R)$. Then we have the following weights on homogenous elements of $Q(R) \mathbin{_s\otimes_p} Q(R)$:

\begin{equation}\label{equation: tridegleft}\textrm{deg}(r\otimes 1) = (a, b ,0)\end{equation}  \begin{equation}\label{equation: tridegright} \textrm{deg}(1\otimes r) = (0,a ,b).\end{equation} The $\Gm^3$-action is such that the map\begin{displaymath}
\op{Spec}Q(R)\times_{\op{Spec}R}\op{Spec}Q(R) \to \op{Spec}R\times\op{Spec}R
\end{displaymath} corresponding to $p \otimes s$ is equivariant for the projection $\Gm^3 \to \Gm^2$ onto the first and third factor. The next lemma relates $Q(R) \mathbin{_s\otimes_p} Q(R)$ to $Q(R)$ directly after taking suitable invariants. In what follows, let  $\mathbb{Z}\subset \mathbb{Z}^3$ be the inclusion into the middle factor. If $M$ is a $\mathbb{Z}^3$-graded module, we let $(M)_0$ denote the $\mathbb{Z}^2$-graded sub-module obtained by taking degree zero in the middle factor. Likewise, if $f: M\to N$ is a map of $\mathbb{Z}^3$-graded modules, $(f_0): (M)_0 \to (N)_0$ is the corresponding restriction. 

\begin{lemma} \label{lemma: centrality of eta}
 We have a commutative diagram
 \begin{center}
 \begin{tikzpicture}[scale=1,level/.style={->,>=stealth,thick}]
	\node (a) at (-5,0) {$(Q(R) \mathbin{_s\otimes_p} Q(R))_0$};
	\node (b) at (1,1) {$(Q(R) \mathbin{_s\otimes_\pi} \Delta(R))_0$};
	\node (c) at (1,-1) {$( \Delta(R) \mathbin{_\sigma\otimes_p} Q(R))_0$};
	\node (d) at (5,0) {$Q(R)$};
	\draw[level] (a) -- node[above left] {$\scriptstyle (1 \otimes \eta)_0$} (b) ;
	\draw[level] (a) -- node[below left] {$\scriptstyle (\eta \otimes 1)_0$} (c) ;
	\draw[level] (b) -- node[above left] {$\sim$} (d) ;
	\draw[level] (c) -- node[below left] {$\sim$} (d) ;
 \end{tikzpicture}
 \end{center}
\end{lemma}

\begin{proof}
Both the top and bottom map take $a \otimes b$ to the degree zero piece inside $\Delta(R) \mathbin{_s\otimes_\pi} \Delta(R)$ which happens to land in $Q(R)$.
\end{proof}

\begin{definition} \label{definition: rho}
 Following Lemma~\ref{lemma: centrality of eta}, we set 
$\rho_R$ to be the map
 \begin{equation}\label{equation: rho} 
 \rho_R: (Q(R) \mathbin{_s\otimes_p} Q(R))_0 \to Q(R).
 \end{equation}
given by either $(1 \otimes \eta)_0$ or equivalently $(\eta \otimes 1)_0$. 
\end{definition}

\noindent In very specific situations, the map $\rho_R$ may be an isomorphism. 

\begin{lemma} \label{lemma: sometimes underived Q2 = Q}
 Let $R$ be an object of $\mathsf{CR}_k^{\Gm}$ and assume that either the weights of $R$ are all non-positive or all non-negative. Then $\rho_R$ is an isomorphism.
\end{lemma}

\begin{proof}
 Assume for simplicity that all the weights are non-negative; the proof for non-positive weights is similar. Using Lemma~\ref{lemma: centrality of eta}, it suffices to show that $(\eta \otimes 1)_0$ is an isomorphism. By Example ~\ref{example: nonnegative},  $Q(R)\cong R[u]$ so that $Q(R)$ is flat over $R$ via $p$. We first note that injectivity of  $(\eta \otimes 1)_0$ then follows from flatness and because the functor of invariants is exact. 
To demonstrate surjectivity, we have to exhibit elements of $Q(R) \mathbin{_s\otimes_p} Q(R)$ of middle degree zero that map to $\sigma(r), \pi(r),$ and $u$. They are, respectively, $1 \otimes \sigma(r), \pi(r) \otimes 1,$ and $u \otimes u$. 
\end{proof}

\noindent Note that in the above lemma, one has $Q(R) \mathbin{_s\otimes_p} Q(R)\cong R[u,v]$, and the gradings from Remark ~\ref{remark: gradings} immediately imply that $(Q(R) \mathbin{_s\otimes_p} Q(R))_0 \cong Q(R)$ as $R$-bimodules. The above lemma shows that $\rho_R$ indeed implements this isomorphism.

\begin{example}\label{example: Q2=Q fails}
 However, the map $\rho_R$ is not an isomorphism in general. An elementary counterexample is $R = k[x,y]/(xy)$ where $[x] = 1$ and $[y] = -1$. Letting $z= yu^{-1}$, one has  \[ Q(R) \cong k[x,z,u]/(xz)\] and \begin{displaymath} Q(R)\otimes_RQ(R) \cong k[x,z',u,u']/(xz'uu' ).  \end{displaymath} The element $x\otimes z'$ has middle degree 0 in $Q(R)\otimes_RQ(R)$ and is sent to $0$ in $Q(R)$ under $\rho_R$.  In Section ~\ref{section: main result} we will remedy such a failure by deriving $Q$, and will revisit this particular example again in Example ~\ref{example: derivedexample}.  
\end{example}

\begin{remark}  Note that in the above example $\op{Spec}R$ is singular. We will see in Lemma ~\ref{lemma: rhoforsmooth} that $\rho_R$ is an isomorphism whenever $\op{Spec}R$ is smooth. 
\end{remark} 

\begin{remark} Intuitively, the property of $\rho_R$ being an isomorphism is similar to the characterization of derived open immersions as being finitely-presented ring maps $f: A\to B$ such that $B\overset{\mathbf{L}}{\otimes}_A B\cong B$, see e.g. \cite[Lemma 2.1.6]{TV}. In particular, this suggests we should derive the tensor product in $Q(R) \mathbin{_s\otimes_p} Q(R)$ to obtain a fully-faithful functor. When $\op{Spec}R$ is smooth, we will see in Proposition ~\ref{proposition: smoothtorvanishes} that this tensor product is automatically derived (i.e. its higher Tor's vanish). For singular cases we will implicitly derive this tensor product during the course of Section ~\ref{section: Deriving Q} by deriving the $Q$ functor itself. 
\end{remark} 

\subsection{Bousfield localizations}\label{section: bousfield} In this section we address the fullness of the functor  $\Phi_{Q_+}$ from Equation ~\eqref{equation: Qplusfunctor}. Indeed, we will show more and in Proposition \ref{proposition: now we see some windows} exhibit a semi-orthogonal decomposition of $\op{D}(\op{Qcoh}^{\Gm} \op{Spec}R)$ such that $\Phi_{Q_+}$ gives the inclusion of one of the factors. This semi-orthogonal decomposition will come from a Bousfield localization. 

\begin{definition}
 Let $\mathcal T$ be a triangulated category. A \newterm{Bousfield localization} is an exact endofunctor $L: \mathcal T \to \mathcal T$ equipped with a natural transformation \[ \delta: \op{Id}_\mathcal{T} \to L\]  such that:

\begin{enumerate} 
\item $L\delta = \delta L$ and
\item $L\delta : L \to L^2$ is invertible. 
\end{enumerate} 
 If instead we have an endofunctor $C: \mathcal T \to \mathcal T$ equipped with a natural transformation $\epsilon: C \to 1$ such that \begin{enumerate}
\item $C\epsilon = \epsilon C$ and 
\item $C\epsilon : C^2 \rightarrow C$ is invertible,\end{enumerate} then one calls $C$ a \newterm{Bousfield colocalization}. 
\end{definition} 

\noindent We refer to \cite[Section 4]{Krause} for background on Bousfield localizations and colocalizations for triangulated categories. 

\begin{remark} \label{remark: automatic bous}
If $P, P' \in \op{D}^b(\op{mod}^{\Gm^2}R \otimes_k R)$ and $\delta : P \to P'$ is any map, then one can easily check that $\Phi_P(\delta)(A) = \delta(\Phi_P(A))$ for any $A\in \op{D}(\op{Mod}^{\Gm} R)$. This means that the first condition for being a Bousfield localization or colocalization is automatically satisfied in the setting of Fourier-Mukai functors, provided has a morphism $\delta: \Delta \to P$ or $\epsilon: P \to \Delta$ (recall $\Delta$ is the kernel of the equivariant identity functor). 
\end{remark}

\begin{definition}
Suppose we have maps of endofunctors
 \begin{displaymath}
  C \xrightarrow{\epsilon} \op{Id}_\mathcal{T} \xrightarrow{\delta} L
 \end{displaymath}
of a triangulated category $\mathcal{T}$ such that \begin{displaymath} Cx  \xrightarrow{\epsilon_{Cx}} x \xrightarrow{\delta_x}Lx\end{displaymath} is an exact triangle for any object $x$.  Then $C \to \op{Id}_\mathcal{T} \to L$ is called a 
 a \newterm{Bousfield triangle} for $\mathcal{T}$ if any of the following equivalent conditions are satisfied:
 \renewcommand{\theenumi}{\roman{enumi}}
\begin{enumerate}[label=(\roman*)]
 \item $L$ is Bousfield localization and  $C(\epsilon_x) = \epsilon_{Cx}$,
  \item $C$ is a Bousfield colocalization and $L(\delta_x) = \delta_{Lx}$,
 \item  $L$ is Bousfield localization and $C$ is a Bousfield colocalization.
\end{enumerate}
\end{definition}

\noindent Let us prove the equivalence of the three conditions in the definition. 
\begin{proof}
i) $\Rightarrow$ ii):  Suppose $L$ is Bousfield localization.  Set $x = Ly$.  Then we get a triangle
\[
CLy \xrightarrow{\epsilon_{Ly}} Ly \xrightarrow{\delta_{Ly}} L^2y
\]
and the map $Ly \to L^2y$ is an isomorphism.  Therefore, $CLy = 0$.  Now, consider the triangle
\[
C^2y  \xrightarrow{C(\epsilon_y)}  Cy \xrightarrow{C(\delta_y)} CLy.
\]
Since $CLy =0 $ the first map is an isomorphism as desired.
ii) $\Rightarrow$ i) is by symmetry.
As i) is equivalent to  ii), it is obvious that they are both equivalent to iii).
\end{proof}

\begin{lemma} \label{lemma: bousfield triangle to sod}
Let $C \to \op{Id}_\mathcal{T} \to L$ be a Bousfield triangle for a triangulated category $\mathcal{T}$. Then there is a weak semi-orthogonal decomposition
 \begin{equation} \label{equation: bousfield sod}
  \mathcal T = \langle \op{Im }L , \op{Im }C \rangle. 
 \end{equation} Here $\op{Im}$ denotes the essential image. 
\end{lemma}

\begin{proof}
 By definition, for any object $x$ of $\mathcal T$, we have a triangle
 \begin{displaymath}
  Cx \to x \to Lx
 \end{displaymath}
in $\mathcal{T}$. Let $f: Cx \to Ly$ be any map. We have a commutative diagram
 \begin{center}
 \begin{tikzpicture}[scale=1,level/.style={->,>=stealth,thick}]
	\node (a) at (-1,1) {$C^2x$};
	\node (b) at (1,1) {$CLy$};
	\node (c) at (-1,-1) {$Cx$};
	\node (d) at (1,-1) {$Ly$};
	\draw[level] (a) -- node[above] {$Cf$} (b) ;
	\draw[level] (a) -- node[left] {$\sim$} (c) ;
	\draw[level] (b) -- (d) ;
	\draw[level] (c) -- node[above] {$f$} (d) ;
 \end{tikzpicture}
 \end{center}
 with the left vertical map an isomorphism since $C$ is a Bousfield colocalization. Since $Ly \to L^2y$ is also an isomorphism, its (co)cone $CLy$ must be $0$. Thus, up to composition with an isomorphism,  $f = 0$.
\end{proof}

\begin{lemma} \label{lemma: mixed sod}
Let $C_1 \xrightarrow{\epsilon_1} 1 \xrightarrow{\delta_1} L_1$ and $C_2 \xrightarrow{\epsilon_2} 1 \xrightarrow{\delta_2} L_2$ be Bousfield triangles for a triangulated category $\mathcal{T}$ such that $L_1 C_2 \xrightarrow{L_1(\epsilon_2)} L_1$ is an isomorphism.
Then there is a weak semi-orthogonal decomposition
 \begin{displaymath}
  \mathcal T = \langle \op{Im }C_2 \circ L_1  , \op{Im }C_2 \circ C_1, \op{Im} L_2 \rangle.
 \end{displaymath}
This induces a fully-faithful functor
 \begin{displaymath}
 F : \mathcal T / \op{Im }C_1 \to \mathcal T.
 \end{displaymath}
\end{lemma}
\begin{proof}
For the first statement, by Lemma ~\ref{lemma: bousfield triangle to sod} we only need to further decompose $\op{Im } Q$ in Equation ~\eqref{equation: bousfield sod}.  From Lemma~\ref{lemma: bousfield triangle to sod} and Lemma~\ref{lemma: local cohomology triangle}, we know that for any object $M \in \mathcal T$ there is exact triangle
 \begin{displaymath}
 C_2(C_1 (M)) \to C_2(M) \to C_2 (L_1 (M))
 \end{displaymath}
 so we only need to check semi-orthogonality. Let $N \in  \mathcal T$ be another object and  
\[
  f: C_2(C_1(M)) \to C_2(L_1(N))
 \]
    be any morphism. Using the semi-orthogonality of $\op{Im} C_2$ and $\op{Im} L_2$, $f$ corresponds uniquely to a map \begin{displaymath}g: C_2(C_1 (M)) \to L_1 (N).\end{displaymath} We have a commutative diagram
 \begin{center}
 \begin{tikzpicture}[scale=1,level/.style={->,>=stealth,thick}]
	\node (a) at (-2,1) {$C_2(C_1 (M))$};
	\node (b) at (2,1) {$L_1 (N) $};
	\node (c) at (-2,-1) {$L_1(C_2(C_1 (M)))$};
	\node (d) at (2,-1) {$L_1^2 (N)$};
	\draw[level] (a) -- node[above] {$g$} (b) ;
	\draw[level] (a) -- (c) ;
	\draw[level] (b) -- node[right] {$\sim$} (d) ;
	\draw[level] (c) -- node[above] {$L_1g$} (d) ;
 \end{tikzpicture}
 \end{center}
 Next, we note that 
 \begin{align*}
L_1(C_2(C_1 (M))) & \cong L_1(C_1 (M)) & \text{ by assumption} \\
& \cong 0 & \text{ since } L_1C_1 = 0.
 \end{align*}
 Hence $g =0$ and therefore $f =0$.
 
 For the second part of the statement,
apply \cite[Lemma 1.4]{Orl09} (see also Proposition 4.9.1 of \cite{Krause}) to the case $\mathcal D =  \op{Im} C_2$, $ \mathcal N =  \op{Im} C_2C_1$.  This gives an equivalence 
\[
 \op{Im} C_2L_1 \cong  \op{Im} C_2 /  \op{Im} C_2 C_1.
 \]
   Now apply  \cite[Lemma 1.1]{Orl09} with $\mathcal D = \op{Im} C_2$, $\mathcal D' =   \mathcal T$, $\mathcal N = \op{Im} C_2  C_1$, and $\mathcal N' = \op{Im}  C_1$.  This gives an equivalence
   \[
  \op{Im} C_2  /  \op{Im} C_2 C_1 \cong \mathcal T / \op{Im} C_1.
   \]
 Tracing through these equivalences, we have an equivalence
\[
\mathcal T / \op{Im C_1} \cong  \op{Im} C_2 L_1.
\]
which induces the fully-faithful functor $F$.
\end{proof}

We now show that, under certain conditions, the exact triangle
\begin{equation}\label{equation: Qdeltatriangle} 
Q \xrightarrow{\eta _R}  \Delta  \to \op{cone}\eta_R  \to Q[1]
\end{equation} in $ \op{D}^b(\op{mod}^{\Gm}R \otimes_k R)$ yields a Bousfield triangle for the associated Fourier-Mukai transforms.
\begin{lemma} \label{lemma: Q S bousfield triangle}
Let $R$ be an object of $\mathsf{CR}_k^{\Gm}$. Then the triangle of functors
\[
\Phi_{Q} \xrightarrow{\eta} \op{Id} \to \Phi_{\op{cone}\eta_R }
 \]  is a Bousfield triangle if 
\begin{enumerate}
\item The map $\rho_R: (Q(R) \mathbin{_s\otimes_p} Q(R))_0 \to Q(R)$ is an isomorphism, and
\item $\op{Tor}^R_i(Q(R)_p,Q(R))_s)=0$ for all $i>0$, where the subscripts on $Q(R)$ denote the $R$-module structures given by $p$ or $s$ respectively. 
\end{enumerate}
\end{lemma}
\begin{proof}
By Remark \ref{remark: automatic bous}, the functors form a Bousfield triangle if and only if
\[
(Q \overset{\mathbf{L}}{\mathbin{_s\otimes_p}} Q)_0  \xrightarrow{Q(\eta)}   Q 
 \]
 is an isomorphism.  Since $\op{Tor}^R_i(Q(R)_p,Q(R))_s)=0$ for all $i>0$, $(Q \overset{\mathbf{L}}{\mathbin{_s\otimes_p}} Q)_0 = (Q \mathbin{_s\otimes_p} Q)_0$.  Therefore, the map $Q(\eta)$ is just $\rho_R$, which is an isomorphism by assumption.
\end{proof}

\begin{definition}\label{definition: PropertyP} If $R$ is an object of $\mathsf{CR}_k^{\Gm}$ is such that conditions a) and b) appearing in Lemma ~\ref{lemma: Q S bousfield triangle} are both satisfied, we say that $R$ has \newterm{Property} $\mathtt{P}$. 
\end{definition}
\noindent In particular, Lemma ~\ref{lemma: bousfield triangle to sod} says that if $R$ has Property $\mathtt{P}$, then there is a semi-orthogonal decomposition 
 \begin{equation}\label{equation: S Q sod} 
\op{D}(\op{Qcoh}^{\Gm} \op{Spec}R) = \langle \op{Im } \Phi_Q,  \op{Im }\Phi_{\op{cone}\eta} \rangle   . 
 \end{equation} We can further refine this semi-orthogonal decomposition using \newline Lemma ~\ref{lemma: mixed sod} by considering Bousfield (co)localizations coming from local cohomology. Let $\Gamma_+$ denote the local cohomology of $X=\op{Spec}R$ along $V(I^+)=X^-$. 

\begin{lemma} \label{lemma: local cohomology triangle}
There is a Bousfield triangle
\begin{equation}\label{equation: bousfieldforlocalcohom}
\Gamma_{+} \to \op{Id} \to J_{+}
\end{equation} for $\op{D}^b(\op{Qcoh}^{\Gm}\op{Spec}R)$
where \begin{equation} J_+ := j_* \circ j^\ast\end{equation} and $j : \op{Spec } R \setminus V(I^+) \to \op{Spec }R$ is the inclusion. 
\end{lemma}
\begin{proof}
This is standard.  See Example 1.2 of \cite{HR17} which applies to algebraic stacks and, in particular, our situation. 
\end{proof}

\noindent We then refine our semi-orthogonal decomposition from Equation ~\eqref{equation: S Q sod} as follows. 

\begin{proposition} \label{proposition: now we see some windows} 
Let $R$ be an object of $\mathsf{CR}_k^{\Gm}$ which has Property $\mathtt{P}$. Then there is a semi-orthogonal decomposition 
 \begin{displaymath}
  \op{D}(\op{Qcoh}^{\Gm} \op{Spec}R) = \langle  \op{Im} \Phi_Q \circ J_+,  \op{Im} \Phi_Q  \circ \Gamma_+,  \op{Im}\Phi_{\op{cone}\eta } \rangle. 
  \end{displaymath}
   Furthermore, the functor
   \begin{displaymath}
  \Phi_{Q_+} : \op{D}(\op{Qcoh}^{\Gm} \op{Spec }R \backslash V(I^+)) \to \op{D}(\op{Qcoh}^{\Gm} \op{Spec }R)
 \end{displaymath}
 is fully-faithful.
\end{proposition}

\begin{proof}
We note that $J_+ \circ \Phi_Q = J_+$ since inverting any $r\in R$ of positive weight trivializes $Q$ by Lemma~\ref{lemma: Q is trivial on semi-stable}.  Therefore, we may apply Lemma~\ref{lemma: mixed sod} to obtain the result, noting that the map $F$ in that lemma is exactly $\Phi_{Q_+}$ in this case.
\end{proof}

\section{The smooth case}\label{section: smooth} 

We now study the faithful functor \begin{equation}\Phi_{Q_+} : \op{D}^b(\op{coh}^{\Gm}X^-)\to \op{D}^b(\op{Qcoh}^{\Gm}X) 
\end{equation} from Equation ~\eqref{equation: Qplusfunctor} in the case where $X = \op{Spec}R$ is a {\em smooth} affine scheme with $\Gm$-action. As mentioned in the introduction, the smoothness hypothesis is unreasonably restrictive for the demands of birational geometry, although it does subsume many simple examples and, not surprisingly, permits some dramatic simplifications compared to the general case that we will pursue in Section ~\ref{section: main result}. 

\subsection{Affine space}\label{section: affinespace} We first consider the case where $X = \mathbb{A}^n$; here we study $\Phi_{Q_+}$ via direct calculations. We begin by showing in two lemmas that $X =\mathbb{A}^n$ (equipped with any weights) has Property $\mathtt{P}$, i.e. satisfies the hypotheses of Lemma ~\ref{lemma: Q S bousfield triangle}. 

\begin{lemma}\label{lemma: rhoforaffinespace} 
 Let $R =k[x_1\ldots ,x_n]$ where the $x_i$ are equipped with any degrees. Then the map $\rho_R$ from Definition ~\ref{definition: rho} is an isomorphism. 
\end{lemma} 

\begin{proof} 
This can be verified by an explicit calculation using the gradings  on $Q(R) \mathbin{_s\otimes_p} Q(R)$ from Equations ~\eqref{equation: tridegleft} and ~\eqref{equation: tridegright}.
However, let us instead give a quick proof more in the spirit of arguments we will use later in Section ~\ref{section: main result}. 

We induct on $n$. First, if $n=1$, i.e. $R=k[x_1]$, then the statement of the Lemma follows directly from Lemma ~\ref{lemma: sometimes underived Q2 = Q}. (Alternately, note that the statement for $n=0$ is trivial). We then need to show that for any object $S$ of $\mathsf{CR}_k^{\Gm}$, if $\rho_S$ is an isomorphism then $\rho_{S[x]}$ is also an isomorphism (where $x$ is given any weight). Indeed, by Lemma ~\ref{lemma: Rx}
\[ Q(S[x])\otimes Q(S[x]) \cong Q(S)[y] \otimes Q(S)[y] \] 
and it's easily verified that
\[
(Q(S)[y] \otimes Q(S)[y])_0 \cong (Q(S) \otimes Q(S))_0 [y].
\]
Then, using Lemma ~\ref{lemma: Rx} again,
\begin{align*}
(Q(S) \otimes Q(S))_0 [y] & = Q(S)[y] \\
& = Q(S[x]).
\end{align*}
\end{proof} 

\begin{lemma}\label{lemma: affinespacetor}  Let $R =k[x_1, \ldots ,x_n]$ where the $x_i$ are equipped with any degrees. Then $\op{Tor}^R_i (Q(R)_p, Q(R)_s)=0$ for $i>0$. 
\end{lemma} 
\begin{proof} In the notation of Example ~\ref{example: affinespaceQ}, we have  \begin{displaymath}  
Q(R) = k[\mathbf{x^+}, \mathbf{x^-}, \mathbf{y^+}, \mathbf{y^-}, u]/ \langle \mathbf{y}^+ - u^{a}\mathbf{x}^+,  \mathbf{x}^- -u^{-b}\mathbf{y}^- \rangle \end{displaymath} which exhibits $Q(R)$ as a complete intersection ring inside $R\otimes_k R[u]$ where   \begin{displaymath} R \otimes_k R =k[\mathbf{x^+}, \mathbf{x^-}, \mathbf{y^+}, \mathbf{y^-}].\end{displaymath} Since $R\otimes_k R[u]$ is a flat $R\otimes_k R$-module, the Koszul resolution 
\begin{align*}
\mathcal K^{\bullet} _{R\otimes_k R[u]}( \mathbf{y}^+ - u^{a}\mathbf{x}^+,  \mathbf{x}^- -u^{-b}\mathbf{y}^-).
 \end{align*} 
 of $Q(R)$ as $R\otimes_k R[u]$-module gives a flat resolution of $Q(R)_s$ as an $R$-module.  Hence,
 \begin{align*}
 Q(R)_p \otimes^{\mathbb L}_R Q(R)_s & = Q(R)_p \otimes_R K^{\bullet} _{R\otimes_k R[u]}(\mathbf{y}^+ - u^{a}\mathbf{x}^+,  \mathbf{x}^- -u^{-b}\mathbf{y}^-) \\
 & = K^{\bullet} _{Q(R) \otimes_k R[u] }( \mathbf{y}^+ - u^{a}v^{a}\mathbf{s}^+,  \mathbf{s}^- -u^{-b}\mathbf{y}^-).
 \end{align*}
where we identify \[ Q(R) \otimes_k R[u] = k[\mathbf{s^+}, \mathbf{s^-}, \mathbf{y^+}, \mathbf{y^-}, u,v].\]   Since $ y_i^+ - u^{a_i}v^{a_i}s_i^+,  s_i^- -u^{-b_i}y_i^-$ is a regular sequence (it just solves out certain variables), this complex has no higher homology.
\end{proof} 

\noindent Combining Lemmas ~\ref{lemma: rhoforaffinespace} and ~\ref{lemma: affinespacetor} with Proposition~\ref{proposition: now we see some windows}  immediately gives the following. 

\begin{proposition} \label{proposition: affinefullyfaithful}
 Let $\Gm$ act on $\mathbb{A}^n = \op{Spec}R$. Then  \[ \Phi_{Q_+} : \op{D}^b(\op{coh}^{\Gm}\mathbb{A}^n\setminus V(I^+))\to \op{D}^b(\op{Qcoh}^{\Gm}\mathbb{A}^n) \] is fully-faithful.
\end{proposition} 

We now study the essential image of $\Phi_{Q_+}$ in the affine space case and show that it coincides with the subcategory generated by weights in the interval $(\mu ,0]$ where $-\mu$ is the sum of the positive weights of the $\Gm$-action. To this end, we first establish a useful set of related generating objects.

\begin{lemma}\label{lemma: affinespacegenerators} Equip $R=k[x_1,\ldots ,x_n]$ with any $\Gm$-action and set \begin{equation} \mu = - \sum_{\op{deg}x_j> 0 } \op{deg}x_j .  \end{equation} Then $\op{D}^b(\op{coh} ^{\Gm}\mathbb{A}^n\setminus V(I^+))$ is generated by objects of the form $j^\ast\O(i)$ where $i\in \mathbb{N}$ is in $(\mu,0]$. Here $j: \mathbb{A}^n\setminus V(I^+)) \to \mathbb{A}^n$ is the inclusion and $R(i)$ denotes structure sheaf of $\mathbb{A}^n = \op{Spec}R$ equipped with weight $i$ with respect to the $\Gm$-action. 
\end{lemma}
\begin{proof}
 The Koszul complex on $\{ x_j \ | \ \op{deg}x_j > 0 \}$ is acyclic on $\mathbb{A}^n\setminus V(I^+)$, so the object $j^\ast \O (\mu)$ is generated by such objects.  Similarly, tensoring the Koszul complex by some $j^*\O(t)$, we get that  $j^\ast\O(\mu + t)$  is generated by those $\O(i)$ with $\mu +t < i \leq t$.  By induction, we can thus generate any $j^\ast\O(k)$ with $k \leq 0$ by the objects $j^\ast\O (i)$ with $i\in (\mu,0]$.  Similarly, we can generate an object of the form  $j^\ast \O(i)$ with $i>0$ by twisting the Koszul complex by $j^\ast\O(i)$ and peeling off the top term.

\end{proof} 

\begin{proposition} \label{proposition: window affine space case} With notation as above, the essential image of the equivariant Fourier-Mukai functor
\[
\Phi_{Q_+} : \op{D}^b(\op{coh}^{\Gm} \mathbb{A}^n\setminus V(I^+)) \to \op{D}^b(\op{Qcoh}^{\Gm}\mathbb{A}^n)
\]
is the full subcategory generated by those $R(i)$ such that $i \in (\mu ,0]$. 
\end{proposition}
\begin{proof} 
Since $\Phi_{Q_+}$ is fully-faithful by Proposition~\ref{proposition: affinefullyfaithful}, Lemma ~\ref{lemma: affinespacegenerators} tells us that the essential image of $\Phi_{Q_+}$ is equivalent to the full subcategory of $\op{D}^b\op{Qcoh}(\mathbb{A}^n)$ generated by the objects $\Phi_{Q_+}(j^\ast\O(i))$ with  $i\in (\mu, 0]$.  Indeed, we will show that  $\Phi_{Q_+}(j^\ast\O(i)) = R(i)$ for such objects, thus proving the Proposition. 

With notation as in Example ~\ref{example: affinespaceQ}, write $R = k[\mathbf{x}^+, \mathbf{x}^-]$ and 
\begin{align*}
Q(R) &  = R \otimes_k R[u] / \langle \mathbf{y}^+ - u^{a}\mathbf{x}^+,  \mathbf{x}^- -u^{-b}\mathbf{y}^- \rangle =k[ \mathbf{x}^+,\mathbf{y}^-,u]. 
\end{align*} 
To push forward an object from $\mathbb A^n \backslash V(x^+) \times  \mathbb A^n$ to $\mathbb{A}^n$, we let $\mathcal C^\cdot$ be the \u{C}ech resolution of $\mathbb A^n \backslash V(x^+) \times  \mathbb A^n$ given by inverting the $x_i^+$'s. We then compute: 

\begin{align*}
\Phi_{Q_+}(\O(i)) & = (\mathcal C^\cdot \otimes_{R \otimes_k R} Q(R)_s \otimes_R R(i,0))_{(0,*)} \\
& =  \mathcal C^\cdot \otimes_{R \otimes_k R} (Q(R)_s)_{(i,*)} \\
& =  \mathcal C^\cdot \otimes_{R \otimes_k R} R \otimes_k R[u] / \langle y_i^+ - u^{a_i}x_i^+,  x_i^- -u^{-b_i}y_i^- \rangle. 
\end{align*}
Here, the notation $(i, *)$ refers to the $\Z^2$ grading $\mathcal C \otimes_{R} M \otimes_{R} Q(R)$ inherits as a $R \otimes_k R$-module, and $(i,*)$ means the sum over all degree $(i,d)$ pieces for all $d \in \Z$ (this is the degree restriction which corresponds to taking $\gm$-invariants in the equivariant Fourier-Mukai transform). We thus have reduced the problem to computing the \u{C}ech cohomology of $Q(R)  = k[\mathbf{x}^+, \mathbf{y}^-, u]$ with $\op{deg } x_i^+ = (a_i, 0), \op{deg } y_i^- = (0, b_i)$, and $\op{deg }u = (-1,1)$. This is a well known computation.  Namely, one can check that the cohomology vanishes when the cohomological degree is in $(\mu ,0)$, and in degree zero is given by $u^{-i} k[u^a \mathbf{x}^+, \mathbf{y}^-] = R(i)$.
\end{proof}

\begin{remark} The above proof also exhibits an isomorphism of functors \[ \Phi_{Q_+} \circ j^* = \op{Id}\]
when restricted to the full subcategory generated by those $R(i)$ such that $i\in (\mu ,0]$. Another consequence is that the essential image of $\Phi_{Q_+}$ actually lies in $\op{D}^b\op{coh}(\mathbb{A}^n)$, as promised in Remark ~\ref{remark: cpropremark}. 
\end{remark} 

\subsection{The smooth affine case in general} We now consider the case where a ring $R$ in $\mathsf{CR}_k^{\Gm}$ is such that $\op{Spec}R$ is smooth. To make reductions to affine space case, we will repeatedly make use of the Luna Slice Theorem, to which we refer to \cite{drez} for expository background or \cite{Luna} for the original result. Accordingly, from this point on we assume $k$ is a field. Let us state a version at the level of generality we require; in particular we will only require the theorem near points on the fixed locus, which simplifies the statement. 

\begin{proposition}\label{proposition: luna} 
Let $\Gm$ act on a smooth affine variety $X=\op{Spec}R$, and let $x\in X$ be a fixed point for the action. Then there exists a $\Gm$-invariant affine subvariety $V\subseteq X$ containing $x$, called the {\normalfont \textsf{slice}} at $x$, and a diagram  
\[
\begin{tikzcd}
& V \ar[dr, "f"] \ar[dl, swap, "g"]& \\
T_x X & & X 
\end{tikzcd}
\]
where the maps $f, g$ are strongly \'etale (see Remark ~\ref{remark: stronglyetale} below for a reminder of this definition), and $T_x X$ denotes the tangent space to $X$ at $x$. 

Moreover, it can be arranged so that $V = \op{Spec}R_r$  where $\op{deg}r=0$ and that the image of $V$ under $g$ is $\op{Spec}T_t$ where $\op{deg}t=0$, where we have written $T_x X: = \op{Spec}T$. 
\end{proposition} 

\noindent Let us prove that our version indeed does follow from the version in \cite[Theorems 5.3 and 5.4]{drez}.
\begin{proof}
The only non-trivial differences between ~\ref{proposition: luna} and \cite[Theorems 5.3 and 5.4]{drez} are that we require $g$ to be strongly \'etale and not just  strongly \'etale onto its image, and that we may take $V=\op{Spec}R_r$ as claimed. 

Taking $V$ from \cite[Theorems 5.3 and 5.4]{drez}, we may cover $V$ by open subsets of the form $\op{Spec} R_a$. If some $\op{Spec}R_a$ contains the fixed point $x$, we must have $\op{deg}r=0$.  Since strongly \'etale morphisms base change under localization by degree zero elements we may replace $V$ by $\op{Spec } R_a$ with $\op{deg} a = 0$ and we still know that $g$ is strongly \'etale onto its image.  

Now, similarly cover $\op{im }g$ by open subsets of the form $\op{Spec }T_t$.  Once again, in order to contain the origin (the image of the fixed point), we must have $\op{deg }t =0$.  Hence, if we replace $\op{Spec }R_a$ by $\op{Spec}R_{ag(t)}$, then $g$ is a strongly \'etale map to $\op{Spec }T_t$.  Furthermore as $\op{deg }t=0$, the inclusion into $\op{Spec }T$ is strongly \'etale as well.  Therefore, $g$, as a composition, is strongly \'etale.  
\end{proof}

\begin{remark}\label{remark: stronglyetale} Let $G$ be any reductive group. Recall that if $\phi: R\to S $ is a map of rings, such that $\hat{\phi}: \op{Spec}S\to\op{Spec}R$ is equivariant, one says that $\phi$ is strongly \'etale if the induced map on invariant subrings \begin{displaymath} \phi_G : R^G\to S^G \end{displaymath} is \'etale and there is an isomorphism $S\cong R\otimes_{R^G}{S^G}$ in $\mathsf{CR}_k^G$ where the tensor product is taken with respect to $\phi$ and the inclusion $R^G\subseteq R$. \end{remark} 

We now show that the functor $Q$ satisfies base change for strongly \'etale ring maps. 

\begin{proposition} \label{prop: ready for Luna}
 Let $f: R \to S$ be a strongly  \'etale map of rings with $\Gm$-actions. Then,
 \begin {enumerate}
 \item there is an isomorphism of bimodules 
 \[
S \mathbin{_f\otimes}_s Q(R)  \cong Q(R) \mathbin{_p\otimes}_fS \cong Q(S), 
  \]  
 \item and an isomorphism of functors,
 \[
f^* \circ  \Phi_{Q(R)} \cong  \Phi_{Q(S)} \circ f^*.
 \]
 \end{enumerate}
 \end{proposition} 

\begin{proof} 

We first prove a).  We have the obvious map $Q(R) \otimes_R S \to Q(S)$ given by $Q(f)\otimes 1$; we claim this map is an isomorphism. This map sits as the top arrow in the commutative diagram

\[
\begin{tikzcd}
Q(R) \mathbin{_p\otimes}_f S  \arrow[d] \arrow[rr, ] &  &  Q(S)  \arrow[d]\\
   R[u,u^{-1}] \mathbin{_\pi\otimes}_f S  \arrow[rr, "="] & &  S[u,u^{-1}]. 
\end{tikzcd}
\] and the left vertical arrow is injective since $f: R\to S$ is \'etale and thus, in particular, is flat. It follows that $Q(f)\otimes 1$ is also injective. 

To demonstrate surjectivity, we use the isomorphism $S\cong R\otimes_{R^G}{S^G}$ afforded by the strongly \'etale condition. Namely, given $s\in S$, we write $s= \sum_i r_i\otimes s_i$ where each $r_i \in R$ and $s_i\in S^{\Gm}$. Let $\sigma : S\to S[u,u^{-1}]$ be the co-action ring map for the $\Gm$-action on $S$. Then \[ \sigma (s) = \sigma ( \sum r_i\otimes s_i ) = \sum \sigma (r_i)\otimes s_i. \] Now, write each $r_i\in R$ as a sum of homogenous elements: $r_i= \sum r_i^j$ where $\op{deg}(r_i^j) :=n_{ij}$. Then \[ \sigma (r_i) = \sum_j r_i^ju^{n_{ij}}\] and so \[ \sigma (r_i)\otimes s_i  = \sum_j r_i^ju^{n_{ij}}\otimes s_i = \sum_ju^{n_{ij}}\otimes f(r_i^j)s_i \] since the tensor product is over $R$ via $f$ on the right. To show that $\sigma (s)$ is in the image of $Q(f)\otimes 1$, it suffices to show that if $u^k\in Q(R)$ for some $k\in\mathbb{Z}$, then $u^k\in Q(S)$. But this is clear as $Q(f): Q(R)\to Q(S)$ is obtained, by definition, by restricting the map $R[u,u^{-1}]\to S[u,u^{-1}]$ induced by $f$.  
 It is trivial that any $s\in S$ is in the image of $Q(f)\otimes 1$, and so we have shown that any element of the form $s$, $\sigma (s)$, and also the element $u$ are all in the image, but such elements generate $Q(S)$ by definition. 
 
Statement b) is the following chain of equalities:
\begin{align*}
f^* (\Phi_{Q(R)}(M)) & = S \otimes_R (Q(R) \otimes_R M )_{(*,0)} & \\
& = S_0 \otimes_{R_0} (Q(R) \otimes_R M )_{(*,0)} & \text{ since } f \text{ is strongly \'etale} \\
& = (S_0 \otimes_{R_0} Q(R) \otimes_R M )_{(*,0)} & \\
& = (S \otimes_{R} Q(R) \otimes_R M )_{(*,0)} &  \text{ since } f \text{ is strongly \'etale}  \\
& = (Q(S) \otimes_R M )_{(*,0)} &  \text{ by part a) }  \\
& = \Phi_{Q(S)}(f^*M). &   \\
\end{align*}\end{proof} 

\begin{remark} In the above proposition, it is necessary that the map $f: R\to S$ is strongly \'etale and not just \'etale. For example, $k[x]\to k[x,x^{-1}]$ with $\op{deg}x>0$ is an open immersion, but it is easy to verify that base change for $Q$ does not hold. 
\end{remark}

\begin{lemma}\label{lemma: rhoforsmooth} Let  $R$ be an object of $\mathsf{CR}_k^{\Gm}$ such that $X=\op{Spec}R$ is smooth.  Then the map $\rho_R$ from Definition ~\ref{definition: rho} is an isomorphism. 
\end{lemma} 

\begin{proof}
It suffices to prove the map is locally an isomorphism.  We know that $Q(R) \cong \Delta(R)$ away from the contracting locus, hence $\rho_R$ is an isomorphism over $\op{Spec }R \backslash V(I^+)$.  
Now, for each point in the fixed locus one obtains $\gm$-invariant affine open neighborhood produced by the Luna Slice theorem ~\ref{proposition: luna}.  These neighborhoods cover the contracting locus so it remains to check that $\rho_R$ is an isomorphism upon restriction to each such neighborhood.  But $\rho_R$ respects base change for strongly \'etale morphisms by Proposition~\ref{prop: ready for Luna}, so the Luna Slice Theorem reduces us to the case of affine space,  which was Lemma~\ref{lemma: rhoforaffinespace}.  \end{proof}

\begin{proposition}\label{proposition: smoothtorvanishes} 
Let  $R$ be an object of $\mathsf{CR}_k^{\Gm}$ such that $\op{Spec}R$ is smooth. Then $Q(R)$ with its $R$-module structure induced by $p$ is Tor independent of $Q(R)$ with its $R$-module structure induced by $s$.  That is, \[ \op{Tor}^{R}_p(Q(R)_p , Q(R)_s ) =0\] for all $p>0$. 
\end{proposition}

\begin{proof}
The proof is similar to that of the previous lemma. Namely, by Lemma ~\ref{lemma: Q is trivial on semi-stable},  $\Phi_{Q(R)} = \op{Id}$ away from $X^-$, and since $Q$ base changes under all maps in the Luna Slice Theorem by Proposition~\ref{prop: ready for Luna}, we are reduced to the case where $\op{Spec } R = \A^n$, where the statement was proved in Lemma ~\ref{lemma: affinespacetor}.
\end{proof}

\begin{definition}
Let  $R$ be an object of $\mathsf{CR}_k^{\Gm}$ and set $\mu$ to be the sum of the weights of the conormal bundle of $\op{Spec}R/I^+ = X^-$ in $X=\op{Spec}R$.  (Note that $X^-$ is smooth since $X=\op{Spec}R$ was assumed smooth, see e.g. ~\cite[Proposition 1.4.20]{Drinfeld}.) Assume that the fixed locus is connected.  The \newterm{grade restriction window}, denoted by $\weezer$, is the full subcategory of  $\op{D}^b(\op{coh}^{\Gm} \op{Spec }R)$ consisting of objects $A$ such that for any fixed point $x$ and some affine \'etale slice $V = \op{Spec }S$ at $x$, the restriction 
\[
A \otimes_R S \in \op{D}^b(\op{coh}^{\Gm}\op{Spec }S)
\]
is generated by $S(i)$ for $i \in (\mu ,0]$.
  \end{definition}

\begin{lemma} \label{lemma: affinecasewindow}
Let  $R$ be an object of $\mathsf{CR}_k^{\Gm}$ such that $X=\op{Spec}R$ is smooth. Assume that the fixed locus is connected. 
Then the essential image of Fourier-Mukai functor
\[
\Phi_{Q_+} : \op{D}^b(\op{coh}^{\Gm}X\setminus V(I^+)) \to \op{D}^b(\op{Qcoh}^{\Gm}X)\]
lies in $\weezer \subseteq \op{D}^b(\op{coh}^{\Gm}X)$. 
Furthermore, on $\weezer$ there is an isomorphism of functors
\begin{equation} \label{equation: identity on window}
(\Phi_{Q_+}  \circ j^*)_{| \weezer}  = \op{Id_{\weezer}}
\end{equation}
where $j : X\setminus V(I^+) \to X$ is the inclusion.

\end{lemma}
\begin{proof}
We know that $\Phi_Q$ satisfies base change for strongly \'etale morphisms by Proposition~\ref{prop: ready for Luna}.  Furthermore, since semi-stable loci are preserved under strongly \'etale morphisms (\cite[Lemma 3.2.1]{SvdB}), $\Phi_{Q_+}$ also satisfies base change.  Furthermore, since $\weezer$ is defined locally,  to show that $\op{Im}\Phi_{Q_+}\subseteq\weezer$, it suffices to verify that the essential image of $\Phi_{Q_+}$ lands in $\weezer$ locally.

Cover $\op{Spec }R$ by open affine $\Gm$-invariant neighborhoods $V  = \op{Spec }S$ of the fixed locus produced by the Luna Slice Theorem, and let $g: V\to T_xX := \op{Spec}T$ be the strongly \'etale map. If $\op{Spec}L$ is any $\Gm$-invariant subvariety of $\op{Spec}R$, let  $\weezer_L$ be the grade restriction window on $\op{Spec }L$ and $j_L :\op{Spec }L^\setminus\op{Spec}L^-\to \op{Spec }L$ be the inclusion.

We know that $\Phi_{Q_+(T)}$ lands in $\weezer_T$ by Proposition~\ref{proposition: window affine space case}. Furthermore, $ \op{D}^b(\op{coh}^{\Gm}V\setminus V^-)$ is generated by $j^*_S S(i) = g^* j^*_T T(i)$ for all $i\in\mathbb{Z}$. Hence, $\Phi_{Q_+(S)}$ lies in $\weezer_S$. Thus $\Phi_{Q_+}$ lands in $\weezer$ locally, as desired. 

To prove the latter statement of the lemma, suppose $M \in \weezer$, then we can cover $\op{Spec }R$ by open affine $\Gm$-invariant neighborhoods $V  = \op{Spec }S$ of the fixed locus produced by the Luna Slice Theorem where $M|_V$ is generated by $S(i)$ for $\mu < i \leq 0$, as above.  
We have that
\begin{align*}
g^*|_{\weezer_T} & = g^*  \circ (\Phi_{Q_+(T)}  \circ j_T^*)|_{\weezer_T}  & \text{ by Proposition~\ref{proposition: window affine space case}} \\
& =  (\Phi_{Q_+(S)}  \circ j_S^*)|_{\weezer_S} \circ g^*  & \text{ by strongly \'etale base change.} \\
\end{align*}
Since the generators of $\weezer_S$ lie in the essential image of $g^*$, this implies
\[
(\Phi_{Q_+(S)}  \circ j^*)|_{\weezer_S}  = \op{Id_{\weezer_S}}
\]
Hence, we have shown this isomorphism locally on $\op{Spec }R$.

Now, we know that $j^*, j_*$ and $\Phi_Q$ satisfy base change for strongly \'etale morphisms by Proposition~\ref{prop: ready for Luna}.  Furthermore the projection formula gives an isomorphism of functors
\[
\Phi_{Q_+}  \circ j^* \cong \Phi_{Q}  \circ j_* \circ j^*.
\] 
Hence $\Phi_{Q_+}  \circ j^*$ satisfies base change for strongly \'etale morphisms.  There are two natural morphisms
\[
\begin{tikzcd}
\Phi_{Q} \ar[r] \ar[d] &  \op{Id} \\
\Phi_{Q}   \circ j_* \circ j^* &
\end{tikzcd}
\]
The vertical arrow is the unit of the adjunction and the horizontal arrow comes from the map $Q \to \Delta$.  We have checked that both maps are (locally) isomorphisms on $\weezer$ and the result follows.
\end{proof}

\begin{theorem}\label{theorem: smoothcasesummary}
Let  $\op{Spec }R$ be a smooth affine variety with a $\gm$-action and connected fixed locus.   The functor 
\[
\Phi_{Q_+} : \op{D}^b(\op{coh}^{\Gm}\op{Spec }R \setminus V(I^+)) \to \weezer
\]
is an equivalence of categories.
\end{theorem}
\begin{proof}
\noindent Combining Lemmas ~\ref{lemma: rhoforsmooth} and ~\ref{proposition: smoothtorvanishes}  with Proposition ~\ref{proposition: now we see some windows}  gives that $\Phi_{Q_+}$ is fully-faithful.  Essential surjectivity follows immediately from the isomorphism 
\[
(\Phi_{Q_+}  \circ j^*)|_{\weezer}  = \op{Id_{\weezer}}
\]
which is part of Lemma~\ref{lemma: affinecasewindow}.
\end{proof}

\begin{remark}
The assumption of a connected fixed locus can be removed by putting more care into the definition of $\weezer$.  Namely, one needs to keep track of the parameter $\mu$ for each connected component of the fixed locus.  
\end{remark}

\begin{corollary}\label{corollary: affinewc}
Let  $\op{Spec }R$ be a smooth affine variety with a $\gm$-action and connected fixed locus.  Let $\mu_\pm$ be the sum of the weights of the conormal bundle of $\op{Spec}R/I^\pm$ in $X=\op{Spec}R$.  
Define 
\[
\Phi^{wc} := j^*_- \circ ( - \otimes \O(-\mu_+ - 1)) \circ \Phi_{Q_+}
\]
 where $j_- : \op{Spec }R \setminus V(I^-) \to \op{Spec }R$ is the inclusion.
If $\mu_+ + \mu_- = 0$ then
\[
\Phi^{wc} : \op{D}^b(\op{coh}^{\Gm}\op{Spec }R \setminus V(I^+)) \to \op{D}^b(\op{coh}^{\Gm}\op{Spec }R \setminus V(I^-))
\]
is an equivalence of categories.
\end{corollary}
\begin{proof}
By the previous theorem, $\Phi_{Q_+}$ gives an equivalence 
\[
\op{D}^b(\op{coh}^{\Gm} \op{Spec }R \setminus V(I^+)) \cong \weezer_+.
\]
  On the other hand, we may regard $\op{Spec }R$ with the inverse $\gm$-action to produce an isomorphic stack.  This exchanges $I^+$ with $I^-$ so we get an equivalence  
  \[
  \op{D}^b(\op{coh}^{\Gm}\op{Spec }R \setminus V(I^-)) \cong \weezer_-.
  \]
    Under these identifications, the assumption that $\mu_+ = \mu_-$ ensures that 
    \[
     \weezer_+ \otimes \O(-\mu_+ - 1) = \weezer_-.
    \]
      The result follows as $j^*_-$ is the inverse to $\Phi_{Q_-}$ on $\weezer$.
\end{proof}

\section{The general case: homotopical methods}\label{section: simplicialQ}

\subsection{Deriving Q}\label{section: Deriving Q} 

We let $\mathsf{sCR}_k^{\Gm}$ denote the category of simplicial commutative rings with $\Gm$-actions over $k$, i.e. the category of simplicial objects in the category of $\mathbb{Z}$-graded commutative rings. We may refer to a ring $R$ in $\mathsf{CR}_k^{\Gm}$ as an ordinary ring when it becomes necessary to emphasize that it is not a simplicial ring.  For simplicial sets, we will denote $n$-simplices by $\Delta [n]$, the union of their faces by $\partial\Delta[n]$, and let $\Lambda^i[n]$ denote the the simplicial horn, which we recall is the is the union of all of the faces of $\Delta [n]$ except for the $i$-th face. 

\begin{remark}
We will denote elements of  $\mathsf{sCR}_k^{\Gm}$ by $R_\bullet$, and for $n\in\mathbb{Z}$ we denote the ordinary ring structure at level $n$ with respect to the underlying simplicial set by $R_n$. The reader is allowed to be concerned about a potential clash in notation with the grading on a $\Z$-graded ring (for example, each $R_n$ as above is itself $\mathbb{Z}$-graded). We have made efforts to ensure that no explicit $\mathbb{Z}$-gradings are referred to in this section, and so the reader should henceforth assume that all subscripts refer to a simplicial level, unless told otherwise. 
\end{remark}

Recall that $\mathsf{sCR}_k^{\Gm}$ has distinguished objects which play the role of free objects. Namely, let \[ F: \mathsf{sSet} \to \mathsf{sCR}_k\]  be the left adjoint to the forgetful functor from simplicial rings to simplicial sets (here $F$ stands for ``free'' and not ``forget''). Given a simplicial set $X$ and a weight $a \in \Z$, we likewise have an object of $\mathsf{sCR}_k^{\Gm}$ denoted $F(X)^a$, where we declare the degrees of the generators (with respect to the $\mathbb{Z}$-grading) of $F(X_n)^a$ to all be $a$.

\begin{proposition} \label{proposition: model structure}
 The category $\mathsf{sCR}_k^{\Gm}$ possesses a simplicial cofibrantly generated model structure where:
\begin{itemize}
\item the generating cofibrations are the maps $F(\partial \Delta[n])^a \to F(\Delta[n])^a$ for $a \in \Z$ and $n \geq 0$. 
\item The generating trivial cofibrations are $F(\Lambda^r[n])^a \to F(\Delta[n])^a$ for $a \in \Z, n \geq 0$, and $0 \leq r \leq n$.
\item The weak equivalences are those of the underlying simplicial sets, i.e. a map is a weak equivalence if and only if it is a weak equivalence after applying the forgetful functor to simplicial sets. \end{itemize} 
\end{proposition}

\begin{proof}
 This seems to be well-known. For example, it is a special case of \cite[Theorem 9.8]{DHK97} (see the discussion above the theorem for how their result specializes to $\mathsf{sCR}_k^{\Gm}$).  It is also \cite[Example 5.10]{GJ}.
\end{proof}

\begin{remark} A similar statement holds for $\mathsf{sCR}_{k[u]}^{\Gm^2}$.\end{remark} 

The main property we will need is the following. 

\begin{corollary} \label{corollary: what are cofibs}
 Any cofibrant object in $\mathsf{sCR}_k^{\Gm}$ is a retract of a sequential colimit of pushouts along the generating cofibrations $F(\partial \Delta[n])^a \to F(\Delta[n])^a$.
\end{corollary}

\begin{proof}
For cofibrantly generated model categories in general, this fact is known as the small object argument, see e.g. \cite[Theorem, 2.1.14 and Corollary 2.1.15]{Hov01}.
\end{proof}

By applying the functor $Q$ from Section~\ref{section: the functor} level-wise and to all face and degeneracy maps, we obtain a functor which we also denote by $Q$. More precisely, by viewing $\mathsf{sCR}_k^{\Gm}$ as the functor category from the opposite of the simplex category to $\mathsf{CR}_k^{\Gm}$ (and likewise for $\mathsf{CR}_{k[u]}^{\Gm^2}$), we obtain an induced functor 
\begin{equation}\label{equation: Qfunctor} 
 Q: \mathsf{sCR}_k^{\Gm} \to \mathsf{sCR}_{k[u]}^{\Gm^2}.
\end{equation} which by abuse of notation we also denote $Q$. For any $R_\bullet$ an object of $\mathsf{sCR}_k^{\Gm}$, the object $Q(R_\bullet )$ comes equipped with action and projection simplicial ring maps  \begin{equation} p,s : R_\bullet \to Q(R_\bullet ) \end{equation} which agree level-wise with the ordinary action and projection ring maps.

\begin{lemma} \label{lemma: leftquillen} The functor $ Q: \mathsf{sCR}_k^{\Gm} \to \mathsf{sCR}_{k[u]}^{\Gm^2}$ is left Quillen. 
\end{lemma}

\begin{proof} We first show that $Q$ preserves cofibrations and trivial cofibrations. Let $\mathsf{fCR}^a_k$ denote the full subcategory of $\mathsf{CR}^{\Gm}_k$ consisting of free commutative $k$-algebras whose generating set of elements has weight $a$.  By Example ~\ref{example: nonnegative}, $Q(R_n)\cong R_n[u]$ for each $n$ with respect to one of the two module structures depending on the sign of the weights, i.e. \begin{equation}\label{equation: simplicialfree} 
  Q(R_\bullet) \cong R_\bullet [u]
 \end{equation}
 for any object $R_\bullet \in \mathsf{fCR}^a_k$ where $[u] = (a,0)$ if $a \geq 0$ and $[u] = (0,a)$ of $a < 0$. Given this, we see that 
 \[
 QF_k(X)^a \cong
 \begin{cases} F_{k[u]}(X)^{(a,0)} & \text{ if } a \geq 0\\
 F_{k[u]}(X)^{(0,a)} & \text{ if } a < 0
 \end{cases}
 \]
  were $F_k, F_{k[u]}$ denote the free functors for simplicial commutative $k$ and $k[u]$-algebras respectively.  In otherwords there is an isomorphism of functors
  \[
  Q F_k \cong F_{k[u]}.
 \]  It follows that $Q$ preserves the generating cofibrations and the generating trivial cofibrations from ~\ref{proposition: model structure}. By \cite[Lemma 2.1.20 ]{Hov01}, this implies that that $Q$ preserves all cofibrations and trivial cofibrations. 

It remains to show that $Q$ admits a right adjoint. In the Appendix in Equation~\eqref{equation: Q adjunction} we will show that \[ Q :  \mathsf{CR}_k^{\Gm}\to\mathsf{CR}_{k[u]}^{\Gm^2}\] has a right adjoint.  As taking simplicial objects is a 2-functor, we have a corresponding adjoint for $Q$ regarded as the induced functor on simplicial objects.  \end{proof}

The above result allows us to define our promised derived replacement of the functor $Q$.
\begin{definition}\label{definition: simplicialQfunctor}  Let 
\begin{equation}
 \mathsf{L}Q: \op{Ho}\big(\mathsf{sCR}_k^{\Gm} \big)\to \op{Ho}\big(\mathsf{sCR}_{k[u]}^{\Gm^2}\big) 
\end{equation}
be the total left derived functor of $Q$. Here $\op{Ho}$ denotes the homotopy category, i.e. the localization of the category $\mathsf{sCR}_k^{\Gm}$ (resp. $\mathsf{sCR}_{k[u]}^{\Gm^2}$) at weak equivalences. 
\end{definition}

\noindent In other words, if $R_\bullet$ is an object of $\mathsf{sCR}_k^{\Gm}$, we have \begin{equation}
\mathsf{L}Q(R_\bullet ) = Q(S_\bullet)
\end{equation} 
where $S_\bullet \to R_\bullet$ is any cofibrant replacement, which is well-defined since $Q$ is a left Quillen functor. That is, if $\op{Cofib}$ denotes a cofibrant replacement functor in $\mathsf{sCR}_k^{\Gm}$ then  \begin{equation} \mathsf LQ := Q \circ \op{Cofib}.\end{equation}

\subsection{Property \texorpdfstring{$\mathtt{P}_{\op{der}}$}{P derived}} \label{section: main result}

We now introduce the analogue of the map $\rho_R$ from Definition ~\ref{definition: rho} and the analogue of Property $\mathtt{P}$ from Definition ~\ref{definition: PropertyP}. The main result of this subsection will be Theorem ~\ref{theorem: derived Q2=Q}, which will show that cofibrant objects have Property $\mathtt{P}_{\op{der}}$. This result has the effect of bypassing the Tor vanishing assumptions in Lemma ~\ref{lemma: Q S bousfield triangle}, which was our original criterion for fully-faithfulness of the functor on derived categories associated to $Q_+$. 

\begin{definition} \label{definition: property B-hole}
 We say that an object $R_\bullet$ of  $\mathsf{sCR}_k^{\Gm}$ has 
 \newterm{Property} $\mathtt{P}_{\op{der}}$ if the map 
\begin{equation}\label{equation: betaR} 
  \beta_{R_\bullet} : (p \otimes_k s)_* Q(R_\bullet ) \overset{\mathbf{L}}{\mathbin{_s\otimes_p}} Q(R_\bullet ) \to  Q(R_\bullet )
\end{equation} 
given by the composition
 \begin{displaymath}
(p \otimes_k s)_* Q(R_\bullet ) \overset{\mathbf{L}}{\mathbin{_s\otimes_p}} Q(R_\bullet ) \to (p \otimes_k s)_*Q(R_\bullet ) \mathbin{_s\otimes_p} Q(R_\bullet ) \overset{\rho_{R_\bullet}}{\to} Q(R_\bullet )
 \end{displaymath}
 is a weak equivalence. Here the first map in the composition is the truncation of the derived tensor product after application of $(p \otimes s)_*$ (with takes middle invariants with respect to the $\Gm^3$-action on the tensor product) and $\rho_{R_\bullet}$ is the map from Definition~\ref{definition: rho} applied level-wise. 

We use the notation $\beta$ (with no subscript) to denote the natural transformation of the two functors $\mathsf{sCR}_k^{\Gm}\to\mathsf{sCR}_{k[u]}^{\Gm^2}$ determined by the left and right side of Equation ~\eqref{equation: betaR}.  In particular, $R_\bullet$ has Property $\mathtt{P}_{\op{der}}$ exactly when $\beta (R_\bullet ) =\beta_{R_\bullet}$ is a weak equivalence. 

\end{definition}

\begin{example}\label{example: nonpos nonneg graded Q2 = Q} 
Let $R_\bullet$ be an object of $\mathsf{sCR}_k^{\Gm}$ such that, at each level $n$, $R_n$ has only non-negative weights (resp. at each level has only non-positive weights) Then  $Q(R_\bullet )$ is level-wise flat over $R_\bullet $ via one of either $s$ or $p$, and so the map
 \begin{displaymath}
  Q(R_\bullet ) \overset{\mathbf{L}}{\mathbin{_s\otimes_p}} Q(R_\bullet ) \to Q(R_\bullet ) \mathbin{_s\otimes_p} Q(R_\bullet )
 \end{displaymath}
 is a weak equivalence, see e.g. \cite[Corollary II.6.10]{Quillen}. Also $\rho_{R_\bullet}$ is a weak equivalence, indeed it is actually an isomorphism level-wise by Lemma~\ref{lemma: sometimes underived Q2 = Q}. It follows that   $R_\bullet$ has Property $\mathtt{P}_{\op{der}}$. 

In particular, the objects $F(\partial \Delta[n])^a$ and  $F(\Delta[n])^a$ of $\mathsf{CR}_k^{\Gm}$ have Property $\mathtt{P}_{\op{der}}$ for any $a\in\mathbb{Z}$ and $n\geq 0$.
\end{example} 

For more explicit examples of $Q(R_\bullet )$ and $\beta_{R_\bullet}$, it is typically more convenient to work with dg-algebras instead of simplicial rings (which one may do via the Dold-Kan correspondence, at least in characteristic zero).

\begin{example}\label{example: derivedexample}
Consider the example $R = k[x,y] / xy$ with $\op{deg }x =1$ and $\op{deg} y = -1$.  Assume that $k$ is a field of characteristic zero.  Example~\ref{example: Q2=Q fails} showed that without deriving this example, $\beta_R$ (i.e. $\rho_R$ from Definition ~\ref{definition: rho}) is not an isomorphism.  However, $R$ has a cofibrant replacement by the dg-algebra $S = k[x,y,e]$ with $d(e) = xy$ where $e$ has homological degree $1$ and weight $0$.  To compute $\beta_S$ we take the degree zero part of 
\begin{align*}
Q(S) \otimes_S Q(S) & = k[x, yu^{-1}, e, u] \otimes_{k[x,y,e]} k[x', y'u'^{-1}, e', u'] \\ 
& = k[x, y'u'^{-1}, e, u, u'].
\end{align*}
The degree zero part is $k[x, y'u'^{-1}, e, uu']$ which is the realization of the isomorphism $Q(S) = (Q(S)  \otimes_S S[u, u^{-1}])_0$.  Hence, we have corrected the failure of $\beta_R$ to be an isomorphism. That is, Property $\mathtt{P}_{\op{der}}$ holds. 
\end{example}

We now prove a general lemma which gives conditions for a natural transformation between model categories to assign cofibrant objects to weak equivalences. 

\begin{lemma} \label{lemma: properties of cofibs}
Let $F,G : \mathcal C \to \mathcal D$ be functors between model categories and $\eta: F \to G$ be a natural transformation.   Assume that
\begin{enumerate}
\item $\mathcal C$ is cofibrantly generated;
\item $\mathcal D$ is a combinatorial model category in the sense of \cite{Dugger};
\item there is an initial cofibrant object $c_0 \in \mathcal C$ and $\eta(c_0)$ is a weak-equivalence;
\item   if $\eta(c)$ is a weak equivalence for some object  $c \in \mathcal C$, then any pushout of $\eta(c)$ along a generating cofibration is a weak equivalence;
\item $\eta$ commutes with sequential colimits.
\end{enumerate}
  Then, $\eta(c)$ is a weak equivalence for any cofibrant object $c \in \mathcal C$.
\end{lemma}

\begin{proof}
By assumption, any cofibrant object is a retract of a sequential colimit of pushouts along generating cofibrations.  Since any natural transformation respects retracts, it suffices to prove that any sequential colimit of pushouts along generating cofibrations is a weak equivalence. This follows from transfinite induction and the assumptions since in a combinatorial model category being a filtered colimit of weak equivalences, is itself a weak equivalence by \cite[Proposition 7.3]{Dugger}.
\end{proof}

We now begin the process of verifying that the hypotheses  of \newline Lemma ~\ref{lemma: properties of cofibs} are satisfied by the natural transformation $\beta$. Only the hypotheses d) and e) are non-trivial to verify. We first show that $\beta $ satisfies condition d).

\begin{lemma} \label{lemma: Q2 = Q pushouts via gen cofibs}
 Assume that $R_\bullet$ has Property $\mathtt{P}_{\op{der}}$ and that we have a map $f: F(\partial \Delta[n])^a \to R_\bullet$. Then the pushout along the natural map $F(\partial \Delta[n])^a \to F(\Delta[n])^a$ has Property $\mathtt{P}_{\op{der}}$. 
\end{lemma}

\begin{proof}
 For notational simplicity let $S := F(\partial \Delta[n])^a$ and $T := F(\Delta[n])^a$. We want to check that
 \begin{displaymath}
  \beta_{T \otimes_S R_\bullet} : (p \otimes s)_* (Q(T \otimes_S R_\bullet) \overset{\mathbf{L}}{\mathbin{_s\otimes_p}} Q(T \otimes_S R_\bullet) ) \to Q(T \otimes_S R_\bullet)
 \end{displaymath}
 is a weak equivalence. 
 
The map $S \to T$ is a cofibration so level-wise it is a retract of a free commutative extension.  By Lemma~\ref{lemma: leftquillen}, the map $Q(S) \to Q(T)$ is then also a cofibration.   In particular, it is level-wise flat. Therefore, the natural map 
  \begin{equation}
  Q(T) \overset{\mathbf{L}}{\otimes}_{Q(S)} Q(R_\bullet) \to Q(T) \otimes_{Q(S)} Q(R_\bullet)
 \label{equation: weak equiv}
 \end{equation}
 is a weak equivalence. 
 
Now, we have the following chain of weak equivalences
 \begin{align*}
& \left( Q(T) \overset{\mathbf{L}}{\otimes}_{T} Q(T) \right) \overset{\mathbf{L}}{\otimes}_{Q(S) \overset{\mathbf{L}}{\mathbin{_s\otimes_p}} Q(S)} \left( Q(R_\bullet) \overset{\mathbf{L}}{\otimes}_{R} Q(R_\bullet) \right)  \\& \cong   (Q(T)\overset{\mathbf{L}} \otimes_{Q(S)} Q(R_\bullet) ) \overset{\mathbf{L}}{\mathbin{_s\otimes_p}} ( Q(T) \overset{\mathbf{L}} \otimes_{Q(S)} Q(R_\bullet) ) \\
  &\cong   (Q(T) \otimes_{Q(S)} Q(R_\bullet) ) \overset{\mathbf{L}}{\mathbin{_s\otimes_p}} ( Q(T) \otimes_{Q(S)} Q(R_\bullet) ) \\
 &\cong  Q(T \otimes_S R_\bullet) \overset{\mathbf{L}}{\mathbin{_s\otimes_p}} Q(T \otimes_S R_\bullet) \\
 \end{align*}
The first weak equivalence above holds because it is an isomorphism for tensor products or ordinary rings, and a well chosen cofibrant replacement functor commutes with taking tensor  products/coproducts (see Proposition 2.3 of \cite{Dugger} which we may apply due to Proposition~\ref{proposition: model structure}).
The second weak equivalence follows from Equation~\eqref{equation: weak equiv}.  For the last equivalence above, we will show in the appendix in Corollary ~\ref{corollary: Qpreservescolimits} that $Q$ preserves arbitrary colimits; in particular, it preserves tensor products. 

Denote the above chain of weak equivalences by $\phi$. Since $\Gm$ is linearly reductive, $(p \otimes_k s)_*$ preserves weak equivalences, and so 
\begin{equation}\label{equation: pushoutweak} 
\beta_{T \otimes_S R_\bullet} \circ  (p \otimes s)_* \phi =   \beta_T \overset{\mathbf{L}}{\otimes}_{\beta_S} \beta_{R_\bullet}.
\end{equation}

Since $\beta_T,\beta_S$ are weak equivalences by Example~\ref{example: nonpos nonneg graded Q2 = Q} and $\beta_{R_\bullet}$ is a weak equivalence by assumption, the right hand side is a colimit of weak equivalences. Hence, by the 2 out of 3 condition for weak equivalences, $\beta_{T \otimes_S R_\bullet}$ is a weak equivalence, as desired. 
\end{proof}

We now verify that $\beta$ satisfies condition e). 

\begin{lemma} \label{lemma: colimits for Q2=Q}
The natural transformation $\beta$ from Definition ~\ref{definition: property B-hole} commutes with sequential colimits. 
\end{lemma}
\begin{proof}
The functor $Q$ preserves colimits by Corollary~~\ref{corollary: Qpreservescolimits} (this applies to simplicial objects since taking simplicial objects is a 2-functor).  Furthermore, colimits commute with coproducts.  The result follows.
\end{proof} 

\begin{theorem} \label{theorem: derived Q2=Q}
 Let $R_\bullet$ be a cofibrant object of $\mathsf{sCR}_k^{\Gm}$. Then $R_\bullet$ has Property $\mathtt{P}_{\op{der}}$. 
\end{theorem}

\begin{proof}
This is a direct application of Lemma~\ref{lemma: properties of cofibs}.  Namely, we set $\mathcal C = \mathsf{sCR}_k^{\mathbb{G}_m}$,  $ \mathcal D = \mathsf{sCR}_{k[u]}^{\Gm^2}$ and $\eta = \beta$.  The initial object is $k$ which is cofibrant and trivially satisfies Property $\mathtt{P}_{\op{der}}$.  The remaining conditions are verified  by Proposition~\ref{proposition: model structure}, Corollary~\ref{corollary: what are cofibs}, Lemma~\ref{lemma: Q2 = Q pushouts via gen cofibs}, and Lemma~\ref{lemma: colimits for Q2=Q}.
\end{proof} 

\subsection{Base change and recovery of the smooth case}\label{section: derivedsmoothcase} 
In this section we address the difference between $Q(R)$ and $\mathsf{L}Q(R)$ when $R$ is an object of $\mathsf{CR}_k^{\mathbb{G}_m}$, i.e. an ordinary $\mathbb{Z}$-graded commutative ring. In particular, in Proposition ~\ref{proposition: smoothtorvanishing} we will exhibit a weak equivalence between them when $\op{Spec}R$ is smooth, so that the general approach of Section ~\ref{section: simplicialQ} effectively reduces to the results of Section ~\ref{section: smooth} under this assumption. We first prove a strongly \'etale base change result which is a derived version of Proposition ~\ref{prop: ready for Luna}.

\begin{proposition}\label{proposition: derived base change} Let $R$ and $S$ be objects of $\mathsf{CR}_k^{\Gm}$ and let $f: R \to S$ be a strongly \'etale graded homomorphism of ordinary rings. Regard $R$ and $S$ as objects of $\mathsf{sCR}_k^{\Gm}$ by viewing them as constant simplicial rings. Then the base change map
\[
\mathsf LQ(R) \overset{\mathbf{L}}{\mathbin{_s\otimes_f}} S \to \mathsf LQ(S)
\]
is a weak equivalence.
\end{proposition}

\begin{proof}
By definition of strongly \'etale, there is an isomorphism of rings
\[
S = R \otimes_{R_0} S_0
\]
where the subscripts here refer to degree with respect to the $\mathbb{Z}$-gradings and not the level (as these rings are not simplicial).  Now, let $R_\bullet \to S_\bullet$ be the image of $R \to S$ under the cofibrant replacement functor on $\mathsf{sCR}_k^{\Gm}$. Since taking the degree $0$ piece at each level preserves weak equivalences, takes generating cofibrations to generating cofibrations, and commutes with sequential colimits, we have a cofibrant replacement $(R_\bullet)_0 \to (S_\bullet)_0$ of $R_0 \to S_0$ in $\mathsf{sCR}_k$.  Since $R_0 \to S_0$ is flat, this gives a weak equivalence of coproducts
\[
S = R \otimes_{R_0} S_0 = R_\bullet \otimes_{(R_\bullet)_0} (S_\bullet)_0
\]
Now apply $Q$ to get weak equivalences
\begin{align*}
\mathsf LQ(S) & = Q(R_\bullet \otimes_{(R_\bullet)_0} (S_\bullet)_0) \\
& = Q(R_\bullet) \otimes_{(R_\bullet)_0} (S_\bullet)_0 \\
& = \mathsf LQ( R )  \otimes_{(R_\bullet)_0} (S_\bullet)_0 \\
& = \mathsf LQ( R ) \overset{\mathbf{L}}{\otimes_{R_0}} S_0  \\
& = \mathsf LQ( R )\overset{\mathbf{L}}{ \otimes_{R}} S.
\end{align*}
\end{proof}

\begin{corollary} \label{corollary: strongly etale pis}
If $R \to S$ is strongly \'etale, then there is an isomorphism of S-modules
\[
\pi_i(\mathsf L Q(S) ) = \pi_i(\mathsf L Q(R) ) \mathbin{_s\otimes_R}  S.
\]
\end{corollary}

\begin{proof}
This follows from a Quillen spectral sequence (\cite[Theorem 6 (c), Section 6.8]{Quillen} ). Namely, we have an $E^2$ page
\[
E_{pq}^2 = \pi_p(\pi_q(\mathsf L Q(R)) \otimes^{\mathbb L}_R S ) \Rightarrow \pi_{p+q}(\mathsf L Q(R) \otimes^{\mathbb L}_R S ) = \pi_{p+q}(\mathsf L Q(S) ).
\]
Since $R,S$ are constant simplicial rings and $S$ is in particular flat over $R$, this forces $p=0$ and the spectral sequence degenerates.
\end{proof}

Recall from Lemma ~\ref{lemma: Q is trivial on semi-stable} that, for ordinary rings with a $\Gm$-action, $Q(R)$ and $\Delta (R)$ become identified away from the contracting locus, i.e. after localizing by elements of positive degree. This intuitively suggests that $Q$ does not require deriving after taking such a localization. We formulate this intuition more precisely as follows. 

\begin{lemma} \label{lemma: LQ away from contract} If $R$ is an object of $\mathsf{CR}_k^{\Gm}$ and $r\in R$ has positive degree, then there is a weak equivalence
\[
\mathsf L Q(R) \mathbin{_s\otimes_R} R_r = Q(R) \mathbin{_s\otimes_R} R_r = \Delta (R) \mathbin{_\sigma\otimes_R} R_r . 
\]
\end{lemma}
\begin{proof} The second equality is an isomorphism and is just Lemma ~\ref{lemma: Q is trivial on semi-stable}. For the first weak equivalence, notice that the inclusion $Q(R)\to \Delta(R)$ gives to a short exact sequence of simplicial $R[u]$-modules:
\[
0 \to \mathsf L Q(R) \to \mathsf L \Delta(R) \to \mathsf L  (\Delta(R)  / Q(R)) \to 0.
\]
Applying $ (- \otimes_{k[u]} k[u,u^{-1}])$ annihilates $ \mathsf L  (\Delta(R)  / Q(R))$ giving an isomorphism 
\[
 \mathsf L Q(R) \otimes_{k[u]} k[u,u^{-1}] \to \mathsf L \Delta(R) \otimes_{k[u]} k[u,u^{-1}].
\]
In particular, we have an isomorphism of $R[u]$-modules
\begin{align*}
\pi_i(\mathsf L Q(R) \otimes_{k[u]} k[u,u^{-1}] )& = \pi_i( \mathsf L \Delta(R) \otimes_{k[u]} k[u,u^{-1}] ) \\
& = \pi_i( \Delta(R) ) \\
& = \begin{cases}
R[u,u^{-1} ] & \text{ if }  i = 0 \\
0 & \text{ if }  i > 0.
\end{cases}
\end{align*}
Now, under the isomorphism
\[
\pi_0(\mathsf L Q(R) \overset{\mathbf{L}}{\mathbin{_s\otimes_R}} R_r) = \pi_0(\mathsf L Q(R) ) \mathbin{_s\otimes_R} R_r,
\]
the element $u \otimes 1$ becomes a unit since $r$ has positive degree (recall that $s(r)= ru^{\op{deg}r})$.  Hence,
 
 \[
 \pi_i(\mathsf L Q(R) ) \mathbin{_s\otimes_R} R_r  =
\begin{cases}
R_r [u,u^{-1} ] & \text{ if }  i = 0 \\
0 & \text{ if }  i > 0.
\end{cases}
 \]
 \end{proof}

\begin{proposition}\label{proposition: smoothtorvanishing} Suppose that $R$ is an object of $\mathsf{CR}_k^{\Gm}$ such that $\op{Spec}R$ is smooth. Then there is a weak equivalence 
\[
\mathsf L Q(R) = Q(R).
\]
\end{proposition} 
\begin{proof}
We have a map $\mathsf L Q(R) \to  Q(R)$; we must show that the induced map 
\[
\pi_* \mathsf L Q(R)  \to Q(R)
\]
of $R$-modules is an isomorphism, i.e. that $\pi_i( \mathsf L Q(R) ) = 0$ for $i >0$ and $\pi_0( \mathsf L Q(R) ) = Q(R)$.  This can be done locally on $R$. In particular, by taking open sets coming from \'etale slices, we must exhibit isomorphisms \[ \pi_i(\mathsf L Q(R) ) \otimes_R R_r = Q(R) \otimes_R R_r   \] when $r$ has degree zero (to cover the fixed locus) and when $r$ has strictly positive degree (to cover the contracting locus). 
On the fixed locus, Proposition ~\ref{proposition: luna} gives in particular a strongly \'etale map $f: T\to R_r$ where 
$T$ is a free object of $\mathsf{CR}^{\gm}_k$ and $r$ has degree zero.  Now $\mathsf LQ(T) = Q(T)$
since $T$ is cofibrant when regarded as a constant simplicial ring. Proposition~\ref{proposition: derived base change} and Proposition ~\ref{prop: ready for Luna} then give 
\[
\mathsf LQ(R_r) = Q(T ) \otimes_{T} R_r = Q(R_r) .
\]
In particular,
\[
\pi_i(\mathsf L Q(R_r) ) = \begin{cases}
Q(R_r) & \text{ if } i = 0 \\
0 & \text{ if } i > 0 \\
\end{cases}
\]
On the other hand, Corollary~\ref{corollary: strongly etale pis} says that
\[
\pi_i(\mathsf L Q(R_r) ) = \pi_i(\mathsf L Q(R) ) \otimes_R R_r.
\]
It follows that 
\[
\pi_i(\mathsf L Q(R) ) \otimes_R R_r
= \begin{cases}
Q(R_r)  = Q(R) \otimes_R R_r & \text{ if } i = 0 \\
0 & \text{ if } i > 0 
\end{cases}
\]
Similarly, if $\op{deg}r > 0$,  Lemma~\ref{lemma: LQ away from contract}.
gives
\[
\pi_i(\mathsf L Q(R) ) \otimes_R R_r
= \begin{cases}
Q(R) \otimes_R R_r & \text{ if } i = 0 \\
0 & \text{ if } i > 0.
\end{cases}
\]\end{proof}

\subsection{Localizations and Semi-orthogonal Decompostions in the General Case} \label{section: consequences}

This section develops localizations and semi-orthogonal decompositions associated to objects of $\mathsf{sCR}_k^{\Gm}$ in analogy with the case of smooth commutative rings which we considered in Section ~\ref{section: bousfield}. An important step in this direction is understanding how to interpret $\mathsf{L}Q (R_\bullet )$ as a Fourier-Mukai kernel object. That is, associated to $\mathsf{L}Q(R_\bullet )$ we wish to construct a corresponding object in the homotopy category of simplicial modules  \[\op{Ho}\big(  \mathsf{sMod}^{\Gm^2} (R_\bullet\overset{\mathbf{L}}{\otimes}_k R_\bullet )\big),  \] which we will do in Definition ~\ref{definition: Qderdefinition}. This requires some attention to the model structure on the category of simplicial modules. A reader unconcerned with these details may wish to simply inspect Definition ~\ref{definition: Qderdefinition} and the corresponding semi-orthogonal decomposition in Proposition ~\ref{proposition: S Q sod} and bypass the remainder of the section. 

For the model structure on the category of simplicial modules, we refer directly to \cite[Chapter II.6]{Quillen}. To endow a triangulated structure on the respective homotopy categories we use categories of spectra which are a model for the derived category, see e.g. \cite[Section 8.2]{Jardine} or \cite[Sections 2 and 3]{Schwede}; we will follow the exposition of the latter article very closely. In particular, the next definition is just \cite[Definition 2.1.1 and Definition 2.1.2]{Schwede} applied to simplicial rings with a $\Gm$-action. 

\begin{definition}\label{definition: spectra} Let $R_\bullet$ be an object of $\mathsf{sMod}^{\Gm} R_\bullet$.
 A \newterm{spectrum} $M$ in $\mathsf{sMod}^{\Gm} R_\bullet$ is a collection of objects $M_n$ and maps $\Sigma M_n \to M_{n+1}$ for each degree $n \in \mathbb{N}$ where $\Sigma M_n$ is the suspension of $M_n$. Maps of spectra $M\to N$ are defined to be collections of maps $M_n\to N_n$ such that the obvious squares commute. The category of spectra in $\mathsf{sMod}^{\Gm} R_\bullet$ is denoted by $(\mathsf{sMod}^{\Gm} R_\bullet)^\infty$. 
\end{definition}

In the above $\Sigma$ denotes the suspension endofunctor on $\mathsf{sMod}^{\Gm} R_\bullet $, which is defined since $\mathsf{sMod}^{\Gm} R_\bullet$ is a simplicial model category with a zero object. For later use, we recall that the right adjoint of $\Sigma$ is the loop functor $\Omega$. \cite[Corollary 3.1.4]{Schwede} describes the model category structure on the category of spectra; we do not reproduce the details here for the sake of brevity, and because we only require certain structural properties.  However, we remind the reader that a map of spectra is said to be a \newterm{strict fibration} if it is degree-wise a fibration with respect to the model structure on $\mathsf{sMod}^{\Gm} R_\bullet$. Likewise, a map of spectra is said to be a a \newterm{strict weak equivalence} if it is degree-wise a weak equivalence.  The fibrations (sometimes called stable fibrations) and weak equivalences (sometimes called stable weak equivalences) require additional properties that we will delegate to the proofs below.  
  
Given any object $X$ of  $\mathsf{sMod}^{\Gm} R_\bullet$ we obtain a \newterm{ suspension spectrum}, denoted $\Sigma^\infty X$, by setting $\Sigma X_n  =\Sigma^nX$ and taking the identity maps in each degree $n$.  This induces a suspension functor \begin{equation}\label{equation: suspension functor} \Sigma^\infty : \mathsf{sMod}^{\Gm} R_\bullet \to  (\mathsf{sMod}^{\Gm} R_\bullet)^\infty . \end{equation} One defines an $\Omega$-\newterm{spectrum} to be a spectrum $M$ such that each map $M_n\to \Omega M_{n+1}$ is a weak homotopy equivalence, and a spectrum is said to be \newterm{connective} if, for each $n\geq 1$, the higher homotopies of $\Omega M_n$ vanish. 

The following result, which is \cite[Lemma 2.2.2]{Schwede}, allows us to use spectra to understand the homotopy category of the category of simplicial modules. 

 \begin{proposition}\label{proposition: spectra in smod vs smod}
 The total left derived functor of the suspension functor $\Sigma^\infty$ gives an equivalence between the homotopy category of graded simplicial modules and of the homotopy category of connective spectra. 
 \end{proposition}

\begin{remark} \label{remark: Qder factors through chain complexes}
The triangulated category $\op{Ho}(\mathsf{sMod}^{\Gm} R_\bullet)$ is actually equivalent to a perhaps more familiar triangulated category via the Dold-Kan correspondence.  For simplicity, assume $R=R_\bullet$ is discrete, i.e. all higher homotopies vanish. Then the Dold-Kan correspondence  gives a Quillen equivalence between $\op{sMod}^{\Gm} R$ and  $\op{Mod}^{\Gm}_{\leq 0} R$, the category of chain complexes of $R$-modules vanishing in positive degrees. This induces a Quillen equivalence of spectra \[ (\op{sMod}^{\Gm} R)^\infty \cong (\op{Mod}^{\Gm}_{\leq 0} R)^\infty.\] One can easily identify $\op{Ho}(\op{Mod}^{\Gm}_{\leq 0} R)^\infty$ with connective spectra becoming complexes with homology concentrated in non-positive degrees, see e.g. \cite[Section 8.2]{Jardine}. Hence, if one prefers, one can always in practice work with differential graded modules over a graded commutative dg-algebra for the purposes of this section. We have not done so, though, because our definition of $\mathsf{L}Q$ via simplicial rings (instead of connective chain complexes) sits more naturally in the category of spectra of modules. 
\end{remark}

Up to one important detail we will discuss immediately below, we now have enough structure in place to give a repackaging of $\mathsf{L}Q(R_\bullet )$ as a kernel object. 

\begin{definition}\label{definition: Qderdefinition} Let $R_\bullet$ be an object of $\mathsf{sCR}_k^{\Gm}$ and $S_\bullet \to R_\bullet$ a cofibrant replacement. Then $\mathsf{L}Q(R_\bullet) = Q(S_\bullet)$ is naturally a $\Z^2$-graded simplicial $S_\bullet \otimes_k S_\bullet$-module, and hence gives an object \begin{equation}\label{equation: Qder}Q_{\op{der}}(R_\bullet) := \Sigma^\infty \mathsf{L}Q(R_\bullet)= \Sigma^\infty Q(S_\bullet)\end{equation} of $\op{Ho} (\mathsf{sMod}^{\Gm^2} R_\bullet\otimes_k^{\mathbf{L}}R_\bullet )^{\infty}$ under the identification of homotopy categories in Proposition~\ref{proposition: spectra in smod vs smod}. 
\end{definition} 

\noindent As hinted above, there is a subtlety in the above Definition, in that $Q_{\op{der}}(R_\bullet )$ is not actually well-defined unless the weak equivalence $S_\bullet \to R_\bullet$  induces an equivalence \[ (\mathsf{sMod}^{\Gm^2} S_\bullet\overset{\mathbf{L}}{\otimes}_k S_\bullet )^\infty \to (\mathsf{sMod}^{\Gm^2} R_\bullet \overset{\mathbf{L}}{\otimes}_k R_\bullet)^\infty . \] Indeed, we will show in Proposition ~\ref{proposition: weakly equivalent derived rings have equivalent derived categories} that this is always the case. To this end, we first record a simple lemma.

\begin{lemma}\label{lemma: strictfibrations} Any strict fibration $M \to N$ of $\Omega$-spectra in $(\mathsf{sMod}^{\Gm} R_\bullet)^\infty$ is a fibration.
\end{lemma}
\begin{proof}  Let $S$ denote the endofunctor of $\mathsf{sMod}^{\Gm} R_\bullet $ which is defined before \cite[Definition 2.1.4]{Schwede} (and note that Schwede uses the notation $Q$ for this functor, which is totally different than our use of $Q$). Roughly, $S(M)$ is a weakly equivalent $\Omega$-spectrum associated to $M$. By \cite[Proposition 2.1.5]{Schwede}, $M \to N$ being a fibration is equivalent to it being a strict fibration and the natural map \[ f: M \to S (M) \overset{\op{L}}{\times}_{S (N)} N\]  being a weak equivalence. Since $S(A)$ and $S(N)$ are $\Omega$-spectra, the maps $M \to S(M)$ and $N \to S(N)$ are strict weak equivalences. So we have a diagram
\[
\begin{tikzcd}
& S(M) \times_{S(N)} N \ar[d, "g"] \\
M \ar[r, "h"]  \ar[ur, "f"] & S(M) = S(M) \times_{S(B)} S(N)
\end{tikzcd}
\]
where $g,h$ are weak equivalences. Hence, so is $f$, as desired. 
\end{proof}

\begin{proposition} \label{proposition: weakly equivalent derived rings have equivalent derived categories}
 If $f: R_\bullet^\prime \to R_\bullet$ is a weak equivalence in $\mathsf{sCR}_k^{\Gm}$, then restriction of scalars induces a Quillen equivalence 
 \begin{displaymath}
 \op{Res}:    (\mathsf{sMod}^{\Gm} R_\bullet)^\infty \to (\mathsf{sMod}^{\Gm} R_\bullet^\prime)^\infty .
 \end{displaymath}
 Likewise, restriction of scalars induces a Quillen equivalence 
 \begin{displaymath}
(\mathsf{sMod}^{\Gm^2} R_\bullet \overset{\mathbf{L}}{\otimes}_k R_\bullet)^\infty \to (\mathsf{sMod}^{\Gm^2} R_\bullet^\prime \overset{\mathbf{L}}{\otimes}_k R_\bullet^\prime)^\infty.
 \end{displaymath} 
 \end{proposition}
 
 \begin{proof} First, we show that $\op{Res}$ gives a Quillen adjunction. Extension of scalars is a left adjoint, so we must show that $\op{Res}$ preserves (stable) weak equivalences and (stable) fibrations.  The forgetful functor $\mathsf{sMod}^{\Gm} R_\bullet \to \mathsf{sSet}$ satisfies the remark given after \cite[Corollary 2.1.6]{Schwede}. Namely, the weak equivalences are those inducing isomorphisms on homotopy groups of the underlying spectra. Since the underlying spectra do not change under restriction of scalars, we see that $\op{Res}$ preserves weak equivalences. 
 
We now show that $\op{Res}$ preserves fibrations. Restriction of scalars preserves strict fibrations of spectra and $\Omega$-spectra, as all objects of $\mathsf{sMod}^{\Gm} R_\bullet$ are fibrant, and so $\op{Res } M \to \op{Res }N$ is a strict fibration if $M\to N$ is a fibration. Again using \cite[Proposition 2.1.5]{Schwede}, it remains to show that \[ \op{Res } M \to S' \op{Res} M \overset{\op{L}}{\times}_{S'  \op{Res} N} \op{Res} N\] is a weak equivalence, where $S$ and $S'$ denote the denote the endofunctors of $\mathsf{sMod}^{\Gm} R_\bullet $ and $\mathsf{sMod}^{\Gm} R_\bullet^\prime$ used in the proof of Lemma ~\ref{lemma: strictfibrations}. 
Now consider the following diagram.
\[
\begin{tikzcd}
\op{Res} M \ar[r] \ar[d] & \op{Res}  SM \ar[r] \ar[d] & S' \op{Res}  SM \ar[d] \\
\op{Res}  N \ar[r] & \op{Res}  SN \ar[r] & S' \op{Res}  SN 
\end{tikzcd}
\]
The left square is homotopy cartesian by assumption. Lemma ~\ref{lemma: strictfibrations} shows that the right square is homotopy cartesian. 
Hence the full rectangle is homotopy cartesian.  We likewise have a diagram 
\[
\begin{tikzcd}
\op{Res}  M \ar[r] \ar[d] & S' \op{Res} M \ar[r] \ar[d] & S' \op{Res} SM \ar[d] \\
\op{Res} N \ar[r] & S' \op{Res} N \ar[r] & S' \op{Res} SN 
\end{tikzcd}
\]
From above the full rectangle is homotopy cartesian and the right square is homotopy cartesian by \cite[Proposition 2.1.3(e)]{Schwede}.  Hence, the left square is homotopy cartesian, so $\op{Res }M \to \op{Res }N$ is a fibration, as claimed, and so $\op{Res}$ gives a Quillen adjunction. 

 To see $\op{Res}$ is a Quillen equivalence, it remains to check that we have an equivalence on the homotopy categories. By \cite[Theorem A.1.1]{SSmodules}, the objects $\Sigma^l R_\bullet (a)$ for $l,a \in \Z$ form a compact set of generators for the homotopy category $\op{Ho}(\mathsf{sMod}^{\Gm} R_\bullet)^\infty$ .  Hence, it suffices to show that, in the homotopy category, restriction of scalars is fully-faithful on these objects. But we assumed $R_\bullet \to R_\bullet^\prime$ is a  weak equivalence, and restriction of scalars commutes with suspension and shifts, so we see that restriction of scalars is indeed fully-faithful on this category. We thus have that $\op{Res}$ gives a Quillen equivalence, as claimed. 

For the second part of the statement regarding bimodules, the proof is entirely the same once one observes that \begin{displaymath}
  R_\bullet^\prime \overset{\mathbf{L}}{\otimes}_k R_\bullet^\prime \to R_\bullet \overset{\mathbf{L}}{\otimes}_k R_\bullet
 \end{displaymath}
 is also a weak equivalence since derived tensor products preserve weak equivalences. 
 \end{proof}

Thus, $Q_{\op{der}}(R_\bullet )$ as given in Definition ~\ref{definition: Qderdefinition} is indeed well-defined. Let $\Delta_{\op{der}}(R_\bullet )$ be the object of $\op{Ho}((\mathsf{sMod}^{\Gm^2} R_\bullet^\prime \overset{\mathbf{L}}{\otimes}_k R_\bullet^\prime)^\infty)$ determined by the diagonal object $\Delta (R_\bullet )$, and let $S_{\op{der}}(R_\bullet )$ denote the cone in $\op{Ho}((\mathsf{sMod}^{\Gm^2} R_\bullet^\prime \overset{\mathbf{L}}{\otimes}_k R_\bullet^\prime)^\infty)$ which fits into the exact triangle 
\[
Q_{\op{der}}(R_\bullet ) \to \Delta(R_\bullet) \to S_{\op{der}}(R_\bullet ).
\]
So we get
\begin{align*}
\Phi_{Q_{\op{der}}} :  \op{Ho}(\mathsf{sMod}^{\Gm} R_\bullet)^\infty & \to  \op{Ho}(\mathsf{sMod}^{\Gm} R_\bullet)^\infty  \\
M & \mapsto (M  \overset{\mathbf{L}}{\mathbin{\otimes}}_s Q_{\op{der}})_0
\end{align*}
the corresponding equivariant integral transform on the homotopy category of spectra. We define $\Phi_{S_{\op{der}}}$ similarly. 

\begin{proposition} \label{proposition: S Q sod}
 For any object $R_\bullet$ of $\mathsf{sCR}_k^{\Gm}$, there is a semi-orthogonal decomposition
 \begin{align*}
 \op{Ho}(\mathsf{sMod}^{\Gm} R_\bullet)^\infty & = \langle \op{Im} \Phi_{S_{\op{der}}}, \op{Im} \Phi_{Q_{\op{der}}} \rangle
 \end{align*}
 which preserves connective spectra.
\end{proposition}

\begin{proof}
By Lemma~\ref{lemma: bousfield triangle to sod}, it is enough to show that
 \[
\Phi_{Q_{\op{der}}} \xrightarrow{\eta} \textrm{Id} \xrightarrow{\op{cone}(\eta )} \Phi_{S_{\op{der}}}
 \]
 is a Bousfield triangle, where we recall that $\eta$ is map on functors induced from $Q_{\op{der}} \to \Delta $. Furthermore, by Remark \ref{remark: automatic bous}, it is enough to show that the map 
\[
Q(\eta): (Q_{\op{der}} \otimes_{R_\bullet} Q_{\op{der}})_0 \to Q_{\op{der}}
\] 
is an isomorphism.  This is equivalent to showing that the map
\[
\beta_R:  (p \otimes_k s)_* \mathsf LQ \overset{\mathbf{L}}{\mathbin{_s\otimes_p}} \mathsf LQ  \to \mathsf{L}Q
\]
is an isomorphism.
This follows immediately from Theorem~\ref{theorem: derived Q2=Q}.
\end{proof}

\begin{example} \label{example: SOD} This is a continuation of Examples~\ref{example: Q2=Q fails} and~\ref{example: derivedexample}.  Recall that in these examples, $R = k[x,y]/(xy)$ where $\op{deg}x = 1$ and $\op{deg}y = -1$.  Assume that $k$ is a field of characteristic zero.  We saw that
  \[ 
  Q(R) =k[x,y,u^{-1}y,u]/(xy)\cong k[x,z,u]/(xz).
  \]
and, as a dg-algebra, 
 \[
Q_{\op{der}}(R) = k[x,z,u,e]
 \]
 with $d(e) = xzu$. Regarding this as an $R$-module, we have
   \[
   Q_{\op{der}}(R) = k[x,z,u]/ xzu
   \]
   where $y$ acts by $zu$. It follows that \[ S_{\op{der}} = (k[x,y, u, u^{-1}]/xy)/ (k[x,z,u]/xzu).\] Now, we can regard  $Q_{\op{der}}(R)$ as a quotient of $k[x,y,z,u] / xy$ by the regular element $uz-y$.  This allows one to compute
\[
(Q_{\op{der}}(R) \overset{\mathbf{L}}{\mathbin{_s\otimes_p}} Q_{\op{der}}(R))_0 = Q_{\op{der}}(R).
\]   
In other words
\[
\Phi_{Q_{\op{der}}(R)} \circ \Phi_{Q_{\op{der}}(R)} = \Phi_{Q_{\op{der}}(R)}, 
 \]
which is exactly the property that grants the existence of a semi-orthogonal decomposition. By checking on the the set of generators $R(i), R/y(i), R/x(i)$ for all $i$, it follows that $\op{Im} \Phi_{Q_{\op{der}}}$ is equal to  the full subcategory of $ \op{D}^{\op{b}}(\op{mod}^{\Gm}k[x,y]/(xy))$ generated by $R(i)$ for $i \leq 0$, which we denote by  $\op{Perf}_{\leq 0}$. 
Hence, we get a semi-orthogonal decomposition,
 \begin{displaymath}
 \op{D}^{\op{b}}(\op{mod}^{\Gm}k[x,y]/(xy)) = \langle\op{Perf}_{\leq 0}^\perp, \op{Perf}_{\leq 0} \rangle. 
 \end{displaymath}
It is worthwhile to observe here that $\op{Im} \Phi_{Q_{\op{der}}}$ lies in perfect complexes.  
\end{example}

\subsection{The global case and D-flips }\label{section: globalcase}

We now undertake the the task of globalizing $Q_{\op{der}}$ beyond the affine case and applying this to diagrams coming from flips. The aim of this subsection is to show that, with this globalization, one can use $Q_{\op{der}}$ to give a ``wall-crossing" functor associated to any $D$-flip. As this will require synthesizing many of the previous constructions in a new context of sheaves of modules, let us give a quick overview of this subsection to guide the reader. \begin{itemize}
\item We first recall a well-known construction of Reid and Thaddeus which associates a sheaf of graded $\mathcal{O}$-modules $\mathcal{A}$ to a $D$-flip.

\item As a first step to building a wall-crossing functor associated to the flip, we consider $Q(\mathcal {A})$, essentially a sheafy version of $Q(R)$ from the affine case. Without difficulties we can likewise consider a functor $Q$ on simplicial sheaves of graded $\O$- modules. 

\item Via Blander's model structure on the category of simplicial sheaves, we can likewise consider derived variants $\mathsf{L}Q$ and $Q_{\op{der}}$ just as in the affine case. 
\end{itemize}

We now recall the definition of a $D$-flip. This is a small variation on the definition of a flip in the sense of Mori theory, but which is well adapted to the setting of $\Gm$-actions. 

\begin{definition}\label{definition: Dflip} Let $X^-$, $X_0$ and $X^+$ be varieties over a field $k$ such that: 
\begin{enumerate} 
\item There is a small contraction $f: X^- \to X_0$, i.e. $f$ is a proper birational morphism whose exceptional locus has codimension at least two (in particular, $X_0$ is not $\mathbb{Q}$-factorial).
\item $X^-$ is equipped with a $\mathbb{Q}$-Cartier divisor $D$ such that $\mathcal{O}(-D)$ is $f$-ample. 
\item $X^+$ admits a small contraction $g: X\to X_0$, and the induced birational map $h: X^-\dashrightarrow X^+$ is such that $\mathcal{O}(h_\ast D)$ is $\mathbb{Q}$-Cartier and $g$-ample. 
\end{enumerate}
Then $X^+$ is the $D$-\newterm{flip} of $X^-$ over $X_0$. 
\end{definition}

\noindent It is easy to check that  a $D$-flip is unique if it exists. The following construction is recalled from \cite[Proposition 1.7]{Thaddeus} and relates $D$-flips to graded rings. 

\begin{proposition} \label{prop: identify D-flip sides}
Let $X^+$ be the $D$-flip of $X^-$ over $X_0$.  Fix $n\in\mathbb{N}$ and set
\[
\mathcal A : = \bigoplus_{k \in \Z} \O_{X_0} (knD). 
\] For $n$ sufficiently large there are isomorphisms
\[
X^- = [\underline{\op{Spec}}_{X_0} \, \mathcal A^- / \gm] \] \[ X^+ =  [\underline{\op{Spec}}_{X_0}\, \mathcal A^+ / \gm] ,
\]
i.e.\ these quotient stacks are represented by $X^-, X^+$ respectively.
\end{proposition}

\begin{proof}
We choose $n$ such that $\mathcal A_{\geq 0}$ is generated in degree one. We will show in Lemmas~\ref{lemma: deg1} and \ref{lemma: ampleness} below that it follows that $[\underline{\op{Spec}}_{X_0}\, \mathcal A^+ / \gm]$ is represented by $\underline{\op{Proj}}_{X_0} \mathcal A_{\geq 0}$. Since $X^- \to X_0$ is a small contraction, $\mathcal A_{\geq 0} = \bigoplus_{k \in \N} \O_{X_0} (kD) = \bigoplus_{k \in \N} \O_{X_-} (kD)$.  But $\underline{\op{Proj}}_{X_0} \mathcal A_{\geq 0}$ is the relative projectivization of the relatively ample line bundle $\O_{X^-}(D)$, and is thus isomorphic to $X^-$, as claimed. The other equality holds by symmetry.
\end{proof}

\noindent We now formulate the two lemmas promised in the above proof. 

\begin{lemma} \label{lemma: deg1}
Let $R$ be an object of  $\mathsf{CR}^{\gm}_k$ and suppose that $I^+$ is generated in degree one, i.e by elements in $R_1$. Then the global quotient stack
\[ [\op{Spec }R \setminus V(I^+) / \gm ]\] is represented by the scheme $Y$ that is obtained by gluing the open affine varieties $\op{Spec }(R_{f_\alpha})_0$ along $\op{Spec }(R_{f_\alpha f_\beta})_0$ for a set of generators $\{ f_\alpha \}$ in $R_1$. 
\end{lemma}
\begin{proof}
This is the same as showing that $\op{Spec }R \setminus V(I^+)$ is a $\gm$-torsor over $Y$.  Now, by assumption, $\op{Spec }R \setminus V(I^+)$ is covered by $\op{Spec }R_f$ with $f \in R_1$, so it is enough to check that  $\op{Spec }R_f$ is itself a $\gm$-torsor over $\op{Spec }(R_f)_0$.  Indeed, it is the trivial torsor.  Namely, there is an isomorphism of rings
\begin{align*}
R_f & \to (R_f)_0[u,u^{-1}] \\
g & \mapsto  gf^{-i}u^i 
\end{align*}
for $g \in (R_f)_i$ with the inverse map determined by $u\mapsto f$ and $h\mapsto h$  for $h \in (R_f)_0$.
\end{proof}

\noindent The space $Y$ as constructed in Lemma ~\ref{lemma: deg1}  is a quotient of  $\op{Spec }R \backslash V(I^+)$ by $\gm$, hence is equipped with a line bundle $\O_Y(1)$ coming from the pullback of the map to $[\op{pt} / \gm]$.  Notice that if $V(I^+)$ has codimension at least 2 then there is an isomorphism $\Gamma(Y, \O_Y(i)) = R_i$.

\begin{lemma} \label{lemma: ampleness}
Let $R$ be an object of $\mathsf{CR}^{\gm}_k$ such that $I^+$ is generated in degree one and that $V(I^+)$ has codimension at least two. Then the line bundle $\O_Y(1)$ defined above is ample.
\end{lemma}
\begin{proof} Ampleness is equivalent to the complements of sections forming a cover $Y$ such that the natural map to the coordinate ring is an open immersion.  By assumption, the complements of global sections of $\O_Y(1)$ given by elements $f \in R_1$ cover $Y$.  Hence, it suffices to show that the map
\[
Y \to \op{Proj }(\bigoplus_{i \in \N} \Gamma(Y, \O_Y(i))) = \op{Proj }R_{\geq 0}
\]
is an open immersion.  This is so because an open subset $\op{Spec }(R_f)_0 \subseteq Y$ is homeomorphic to its image $\op{Spec }((R_{\geq 0})_f)_0$, as the natural inclusion of rings $
(R_{\geq 0})_f \to R_f$ is an isomorphism since $f$  is a unit in $R_f$.  Namely, the inclusion is surjective since any element $gf^i$ is the image of $gf^N f^{i-N}$ for $N >> 0$.
\end{proof}

We now want to use $Q_{\op{der}}$ from Equation ~\eqref{equation: Qder} to define a wall-crossing functor for $D$-flips. Of course, $Q_{\op{der}}$ was only defined in the affine case, and so we must sheafify. Let us begin this process first with our original functor \[ Q: \mathsf{CR}_k^{\Gm}\to \mathsf{CR}_{k[u]}^{\Gm^2}\] from Definition ~\ref{definition: Q}. This $Q$ automatically gives a presheaf of $\O_Y$-algebras on the affine site over $\op{Spec } k$. Denote the category of $\mathbb{Z}$-graded $\mathcal \O_Y$-algebras by $\sf{CR}^{\Gm}_{\mathcal \O_Y}$. Sheafifying gives a functor   
\begin{equation}\label{equation: Qsheaf}
 Q : \mathsf{CR}^{\Gm}_{\mathcal O_Y} \to \mathsf{CR}^{\Gm^2}_{\mathcal O_Y[u]}
\end{equation} which, again, by abuse we also denote $Q$. 

\begin{remark} An object $Q(\mathcal{A})$ with $\mathcal{A}$ an object of $\mathsf{CR}^{\Gm}_{\mathcal O_Y}$ is not necessarily quasi-coherent.  This is a minor complication for our main goal, which is to define Fourier-Mukai functor between suitable homotopy categories of quasi-coherent sheaves, and not sheaves of arbitrary $\O$-modules. The following lemma will help us alleviate these concerns. 
\end{remark} 

%
%
%

\begin{lemma} \label{lemma: quasi-coherent agreement}
Let $Y$ be a scheme with trivial $\gm$-action and $\mathcal A$ an object of $\mathsf{CR}^{\Gm}_{\mathcal O_Y}$. Then $Q(\mathcal A)$ is quasi-coherent. 

 In particular, if $Y$ is affine, then $Q(\mathcal A)$ is the sheaf associated to $Q(\mathcal A(Y))$.
\end{lemma}

\begin{proof}
Let $U \subseteq V \subseteq Y$ be affine open subsets.  Then quasi-coherence of $Q(\mathcal A)$ is the same as having  
\begin{align*}
Q(\mathcal A (U)) & = Q(\mathcal A (V) \otimes_{\O_Y(V)} \O_Y(U) ) \\
 & = Q(\mathcal A (V)) \otimes_{\O_Y(V)} \O_Y(U). 
\end{align*}
Indeed, the first equality holds because of quasi-coherence of $\mathcal A$.  The second equality follows since the $\gm$-action on $Y$ is trivial.
\end{proof}

Since the space $X_0$ in a $D$-flip diagram is in practice usually singular, we must derive \[  Q : \mathsf{CR}^{\Gm}_{\mathcal O_Y} \to \mathsf{CR}^{\Gm^2}_{\mathcal O_Y[u]}\]in this sheaf-theoretic setting, analogous to how we derived $Q$ in the affine setting during the course of Section \ref{section: simplicialQ}. Fortunately, we will be able to directly transfer results from Section \ref{section: simplicialQ} with only minor difficulties, once we understand a suitable model structure on simplicial sheaves. To this end, let $\mathsf{sCR}^{\Gm}_{\mathcal O_Y}$ denote the category of simplicial objects in $\mathsf{sCR}^{\Gm}_{\mathcal O_Y}$. The following result is essentially \cite[Theorem 2.1]{Blander}. 

\begin{proposition} \label{proposition: local projective model structure}
 The category $\mathsf{sCR}^{\Gm}_{\mathcal O_Y}$ admits a cofibrantly generated simplicial model structure where the weak equivalences are the maps that induce isomorphisms on the homotopy sheaves, and a set of generating cofibrations are \begin{displaymath} \op{Sym}_{\mathcal O_Y} i_! \mathcal O_U \otimes \partial \Delta[n]^a \to \op{Sym}_{\mathcal O_Y} i_! \mathcal O_U \otimes \Delta[n]^a,\end{displaymath} where $U\to Y$ is any open immersion and $a\in\mathbb{Z}$ is any weight. (Here $i_! \mathcal O_U$ denotes the extension by zero sheaf.)
\end{proposition}

\begin{proof}
 From \cite[Theorem 2.1]{Blander}, and in particular the final remark in the proof, we know that the category $\op{sSh}(Y)$ of simplicial sheaves of sets on $Y$ has a model structure with the weak equivalences as above and a set of generating cofibrations given by \begin{displaymath}i_! \mathcal O_U \otimes \partial \Delta[n] \to i_! \mathcal O_U \otimes \Delta[n]\end{displaymath} We then have a family of free, forgetful adjunctions between $\mathsf{sCR}^{\Gm}_{\mathcal O_Y}$ and $\op{sSh}(Y)$ which induce the claimed model category structure (see e.g., \cite[Lemma 9.1]{DHK97} or \cite[Theorem 5.8]{GJ}). 
\end{proof}

The functor $Q$ from equation ~\ref{equation: Qsheaf} gives a functor (which, as usual, we also denote by $Q$) on simplicial categories 
\begin{equation}\label{equation: sQsheaf}
Q : \mathsf{sCR}^{\Gm}_{\mathcal O_Y} \to \mathsf{sCR}^{\Gm^2}_{\mathcal O_Y[u]}
\end{equation}
in the obvious way. Using the model category structure in Proposition ~\ref{proposition: local projective model structure}, it is easy to verify that this functor $Q$ is still left Quillen, so we can define a derived functor $\mathsf{L}Q$ in total analogy with the affine case from Definition ~\ref{definition: simplicialQfunctor}. 

\begin{definition}\label{definition: sheafyLQ}  Let 
\begin{equation}
 \mathsf{L}Q: \op{Ho}\big(\mathsf{sCR}_{\mathcal O_Y}^{\Gm} \big)\to \op{Ho}\big(\mathsf{sCR}_{\mathcal O_Y[u]}^{\Gm^2}\big) 
\end{equation}
be the total left derived functor of the functor $Q$ in Equation ~\ref{equation: sQsheaf} 
\end{definition} 

In particular, any $\mathsf{L}Q(\mathcal A)$ is an object of 
\[ \mathsf{sMod}^{\Gm^2} (\mathcal A \overset{\mathsf{L}}{\otimes}_{\mathcal O_Y} \mathcal A).  \]
Likewise, given $\mathcal{A}$ an object of $\mathsf{sCR}_{\O_Y}^{\Gm}$, we may consider the category of spectra of sheaves of simplicial $\mathcal{A}$-modules $(\mathsf{sMod}^{\Gm}\mathcal{A} )^\infty$, and define an object $Q_{\op{der}}(\mathcal{A})$ of $\op{Ho}( \mathsf{sMod}^{\Gm^2} (\mathcal A \overset{\mathsf{L}}{\otimes}_{\mathcal O_Y} \mathcal A)) $ and a corresponding Fourier-Mukai functor just as in Definition ~\ref{definition: Qderdefinition}. Let us verify that this process does indeed result in desirable properties; in particular, we want that the kernel object $Q_{\op{der}}(\mathcal A)$ preserves quasi-coherence, and is itself quasi-coherent (at least, up to weak equivalence). 

\begin{proposition}\label{proposition: global LQ is qcoh}
Let $Y$ be a scheme with a trivial $\Gm$-action and $\mathcal A$ an object of $\mathsf{CR}^{\Gm}_{\mathcal O_Y}$ which is quasi-coherent as a $\mathcal O_Y$-module. Then the object $\mathsf{L}Q(\mathcal A)$ of $\op{sMod}^{\Gm^2} (\mathcal A \overset{\mathsf{L}}{\otimes}_{\mathcal O_Y} \mathcal A)$ is locally weakly equivalent to a simplicial quasi-coherent sheaf. Hence the functor 
 \begin{displaymath}
  \Phi_{Q_{\op{der}}} : \op{Ho}(\op{sMod}^{\Gm} \mathcal A)^{\infty} \to \op{Ho}(\op{sMod}^{\Gm} \mathcal A)^{\infty}
 \end{displaymath}
 preserves the full subcategories with quasi-coherent homotopy sheaves. 
\end{proposition}

\begin{proof} Let $U\subseteq Y$ be an affine open subset and write $\underline{\op{Spec}}_U \mathcal A(U) = \op{Spec }R$.  Let $\mathcal A_{|U}$ be the sheaf associated to the $k$-algebra $R$. We need to verify that he sheaf associated to $\mathsf{L}Q(R)$ is weakly equivalent to $\mathsf{L}Q(\mathcal A)_{|_U}$. 

We first verify this for the case where $U=Y$ (so that $Y$ itself is affine). Inspecting Proposition ~\ref{proposition: local projective model structure}, we see that the sheaves associated to the generating cofibrations in $\mathsf{sCR}^{\Gm}_k$ are a subset of generating cofibrations in $\mathsf{sCR}^{\Gm}_{\mathcal O_Y}$. Similarly, a weak equivalence in $\mathsf{sCR}^{\Gm}_R$ sheafifies to a weak equivalence in $\mathsf{sCR}^{\Gm}_{\mathcal O_Y}$. Thus, if we take a cofibrant replacement $S_\bullet \to R$ and sheafify, we still get a cofibrant replacement. Using this resolution to compute $\mathsf{L}Q(\mathcal A)$ yields the conclusion by Proposition~\ref{lemma: quasi-coherent agreement}.

In general, if $U\subset Y$ is an affine open subset, let $S_\bullet$ be a cofibrant replacement of $\mathcal{A}$ in $\mathsf{sCR}^{\Gm}_{\mathcal O_Y}$.  Then $(S_\bullet)_{|U}$ is a cofibrant replacement of $\mathcal A_U$. Hence 
\[ \mathsf{L}Q(\mathcal A)_{|U} = Q(S_\bullet)_{|U} = Q((S_\bullet)_{|U}) = \mathsf{L}Q( \mathcal A_U),\] 
which we have already argued is weakly equivalent to the sheaf associated to $\mathsf{L}Q(R)$. 
\end{proof} 

The Dold-Kan correspondence gives an equivalence between the full subcategory of $\op{Ho}(\op{sMod}^{\Gm} \mathcal A)^{\infty}$ with quasi-coherent homotopy sheaves and \[\op{D}(\op{Qcoh}^{\Gm}  \op{Spec}_Y \mathcal A),\] where the latter is defined to be the full subcategory of \[\op{D}(\op{Mod}^{\Gm} \O_{ \op{Spec}_Y \mathcal A})\] consisting of complexes with quasi-coherent cohomology. 

Let $Q_{\op{der}}^{wc}$ be the induced object on $X^- \times X^+$ via restriction and the isomorphism in Proposition~\ref{prop: identify D-flip sides}. This is a generalization of the kernel object $Q^{wc}$ that was introduced in Equation ~\eqref{equation: Qwc} for the Bondal-Orlov flop equivalence, or the functor $j^*_- \circ \Phi_{Q_+}$ from Corollary ~\ref{corollary: affinewc}. 

\begin{question} \label{main question}
Suppose $X^- \dashrightarrow X^+$ is a $D$-flip which is a $K$-equivalence between two smooth projective $k$-varieties, i.e.\ a flop. Is the functor $\Phi_{Q_{\op{der}}^{wc}}$ an equivalence?
\end{question}
\begin{remark}
We know that $\Phi_{Q_{\op{der}}^{wc}}$ preserves quasi-coherence, so that we can view this as a question about quasi-coherent unbounded complexes. Moreover, the compact objects of this category are precisely the perfect objects \cite[Corollary 2.3]{Nee}.   Since $X^+, X^-$ are smooth by assumption, these coincide with bounded coherent complexes so a positive answer to this question should resolve \cite[Conjecture 5.1]{Kaw} (see \cite{KFlops}).
\end{remark}

\begin{remark} 
 To answer the question affirmatively (at least in the Gorenstein case), it should suffice to prove it in the case where $Y$ is affine (see \cite[Theorem 1.22]{RMdS}).  Furthermore, note that when $Y$ is affine, one only requires the machinery up through Section~\ref{section: main result} to study  $\Phi_{Q_{\op{der}}^{wc}}$.
\end{remark}

\section{Appendix: Relations with Drinfeld's space}\label{section: drinfeld} 

We expand on Remark ~\ref{remark: drinfeld}. There we observed that $Q(R)$ agrees with the affine case of the main construction of \cite{Drinfeld}, although we will soon record one subtle distinction. 

 Let $k$ be a field and $Z$ a (not necessarily affine) $k$-scheme of finite type over a field $k$ and possessing an action by $\Gm$ and an open cover which is $\Gm$-stable, i.e. $\Gm$ acts locally linearly.  Drinfeld defines a fpqc sheaf on the category of $k$-schemes over $\A^1_k$ as follows: for an arbitrary scheme $T$ over $\A^1$ assign the set 
\begin{displaymath}
 \op{Hom}_{\A^1}^{\Gm}( \mathbb{X} \times_{\A^1} T, Z)  
\end{displaymath}
where $\mathbb{X}:=\A^1 \times \A^1$ is equipped with the product map $\mathbb{X} \to \A^1$ and $\mathbb{X}$  has the ``hyperbolic" $\Gm$-action $t\cdot (x,y) = (tx,t^{-1}y)$. Amongst other results, \cite[Section 2.4]{Drinfeld} proves that this functor is representable, and so there exists a scheme $\tilde{Z}$ over $\A^1_k$ such that 
\begin{equation}\label{equation: drinfeldfunctor}
\op{Hom}_{\A^1}(T,\tilde{Z}) = \op{Hom}_{\A^1}^{\Gm}( \mathbb{X} \times_{\A^1} T, Z). 
\end{equation}
Equivalently, this means that we have an adjunction
\[
\adjunction{F  }{\{ k\text{-schemes over }\mathbb{A}^1 \} }{ \{ \gm \text{-schemes}\}} {G}
\]where $F(T) := \mathbb{X} \times_{\A^1} T$,  $G(Z) := \tilde{Z}$ and the actions are locally linear on the category on the right. As in Remark ~\ref{remark: drinfeld}, if $Z= \op{Spec}R$ is affine then so is $\tilde{Z}$, and $\tilde{Z} = \op{Spec}Q(R)$.  Therefore both $F$ above and its adjoint $G$ preserve affine schemes.  Restricting the adjunction above to affine schemes and then taking opposite categories, we get an adjunction
\[
\adjunction{Q}{\mathsf{CR}^{\gm}_k}{ \mathsf{CR}_{k[u]}}{Q^{\op{ad}}_{\op{Drin}}}.
\] where $Q^{\op{ad}}_{\op{Drin}} (S) := k[x,y] \otimes_{k[u]}S$.

However, the reader should be careful here.  Recall that in Lemma~\ref{lemma: Q definition}, we defined $Q$ with the following target
\[
Q : {\mathsf{CR}^{\gm}_k} \to {\mathsf{CR}^{\gm^2}_{k[u]}}
\]
not with target $ \mathsf{CR}_{k[u]}$ as above. This is no inherent contradiction, as \cite[Section 2.1.17]{Drinfeld} constructs a $\Gm^2$-action on any $\tilde{Z}$. However, this does mean that $Q^{\op{ad}}_{\op{Drin}}$ as defined above is {\em not} the correct adjoint of our functor $Q$. We will prove below that $Q$ has the following adjoint.  Given a $\Z^2$-graded $k$-algebra $S$ over $k[u]$ where $\op{deg}u = (-1,1)$, let $Q^{\op{ad}}(S)$ denote the subalgebra of $S \otimes_{k[u]} k[x,y]/(xy)$ generated by 
\begin{align*}
 S_{(i,0)}x^i & \text{ for }i \geq 0\text{ and }\\ S_{(0,-i)}y^{-i} & \text{ for }i <0\end{align*}   Here, $u$ maps to $xy$ (so it is zero), $\op{deg}x = (0,1)$, and $\op{deg}y = (-1,0)$.
Since $Q^{\op{ad}}(S)$ has non-zero degree only in the $(i,i)$ summands, we may regard it as a $\Z$-graded $k$-algebra.

\begin{proposition}
There is an adjunction
\begin{equation} \label{equation: Q adjunction}
\adjunction{Q}{\mathsf{CR}^{\gm}_k}{ \mathsf{CR}^{\gm^2}_{k[u]}}{Q^{\op{ad}}}.
\end{equation}
\end{proposition}
\begin{proof} Let $ S, T \in \mathsf{CR}^{\gm^2}_{k[u]}$ and let $R \in \mathsf{CR}^{\gm}_k$ with  maps
 \begin{align*}
  p : R & \to Q(R) \\
  s : R & \to Q(R).
 \end{align*}
Given any map $f : Q(R) \to S$ in $\mathsf{CR}^{\gm^2}_{k[u]}$ , define
\[ \phi_f: R  \to Q^{\op{ad}}(S) \subseteq S \otimes_{k[u]} k[x,y] /(xy) \] by

 \begin{align*}
  r & \mapsto \begin{cases} 
               f(p(r))\otimes x^{\op{deg}r} \;\;\text{  when } \op{deg}r \geq 0 \\
               f(s(r))\otimes y^{-\op{deg}r} \text{ when } \op{deg}r \leq 0.
              \end{cases}
 \end{align*} 
 Conversely, given a map $g : R \to Q^{\op{ad}}(S)$ in $\mathsf{CR}_k^{\Gm}$ define 
 \[
 \psi_g : Q(R)\to S
 \]
  by defining its values on the generating elements of $Q(R)$ as follows: \begin{align*}
r & \mapsto g(r)_x \text{ if } \op{deg}r \geq 0 \\
s(r) & \mapsto g(r)_y \text{ if } \op{deg}r < 0 \\
 \end{align*} 
 where $g(r)_x$ (resp. $g(r)_y$) is the $x$-component (resp. $y$-component) of $g(r)$ under the decomposition as abelian groups
 \[
 S \otimes_{k[u]} k[x,y] /(xy) \cong S/u[x] \oplus S/u[y]. 
 \] The maps $f\mapsto\phi_f$ and  $g\mapsto\psi_g$ are inverse isomorphisms which give the adjunction.
\end{proof}

Since any functor with a right adjoint preserves colimits, we have the following. 
\begin{corollary}\label{corollary: Qpreservescolimits} 
The functor $Q: \mathsf{CR}_k^{\Gm}\to\mathsf{CR}_{k[u]}^{\Gm^2}$ preserves colimits. 
\end{corollary} 



\begin{thebibliography}{99}

\bibitem[Bal17]{Ball}
Ballard, Matthew Robert. {\em Wall crossing for derived categories of moduli spaces of sheaves on rational surfaces}. Algebr. Geom. 4 (2017), no. 3, 263--280. 

\bibitem[BFK12]{BFK}
Ballard, Matthew; Favero, David; Katzarkov, Ludmil. {\em Variation of geometric invariant theory quotients and derived categories} (2012). To appear in Crelle's Journal.

\bibitem[Bla01]{Blander}
Blander, Benjamin A. {\em Local projective model structures on simplicial presheaves}. $K$-Theory 24 (2001), no. 3, 283--301.

\bibitem[BO95]{BO}
Bondal, Alexey; Orlov, Dmitri. {\em Semiorthogonal decomposition for algebraic varieties}. \href{http:/ \! /arxiv.org/abs/alg-geom/9506012}{arXiv:alg-geom/9506012} (1995).

\bibitem[Bri02]{Bridge}
Bridgeland, Tom. {\em Flops and derived categories. } Invent. Math. 147 (2002), no. 3, 613--632. 

\bibitem[BKR01]{BKR}
Bridgeland, Tom; King, Alastair, Reid, Miles. {\em The McKay correspondence as an equivalence of derived categories}. J. Amer. Math. Soc. 14 (2001), 535--554.


\bibitem[CKL12]{CKL} 
Cautis, Sabin; Kamnitzer, Joel; Licata, Anthony. 
{\em Derived equivalences for cotangent bundles of Grassmannians via categorical sl2 actions. } J. Reine Angew. Math. 675 (2013), 53--99. 

\bibitem[CLS11]{CLS}
Cox, David A.; Little, John B.; Schenck, Henry K. {\em Toric varieties.}
Graduate Studies in Mathematics, 124. American Mathematical Society, Providence, RI  (2011), xxiv+841 pp.

\bibitem[DH98]{DH}
Dolgachev, Igor; Hu, Yi. {\em Variation of geometric invariant theory quotients. } With an appendix by Nicolas Ressayre. Inst. Hautes \'Etudes Sci. Publ. Math. No. 87 (1998), 5--56. 

\bibitem[Dre04]{drez} 
Dr\'ezet, Jean-Marc. 
{\em Luna's slice theorem and applications. } Algebraic group actions and quotients, € Hindawi Publ. Corp., Cairo (2004), 39--
89. 

\bibitem[Dri13]{Drinfeld}
Drinfeld, Vladimir.{ \em On algebraic spaces with an action of $\Gm$. } \href{http:/ \! /arxiv.org/abs/1308.2604}{arXiv:1308.2604} (2013).

\bibitem[Dug01]{Dugger} 
Dugger, Daniel. {\em Combinatorial model categories have presentations.} Advances in Mathematics 164.1 (2001), 177--201.

\bibitem[DHK97]{DHK97}
Dwyer, W., P. Hirschhorn, and D. Kan. \emph{Model categories and more general abstract homotopy theory}, Book in preparation.

\bibitem[GJ99]{GJ}
Goerss, Paul G.; Jardine, John F. {\em Simplicial homotopy theory}. Reprint of the 1999 edition. Modern Birkh\"auser Classics. Birkh\"auser Verlag, Basel, (2009), xvi+510 pp.

\bibitem[HR17]{HR17}
Hall, Jack, and David Rydh. \emph{The telescope conjecture for algebraic stacks.} Journal of Topology 10.3 (2017), 776--794.

\bibitem[HL15]{HL}
Halpern-Leistner, Daniel. 
{\em The derived category of a GIT quotient.} J. Amer. Math. Soc. 28 (2015), no. 3, 871--912. 

\bibitem[HS16]{HLSam}
Halpern-Leistner, Daniel; Sam, Steven. {\em Combinatorial constructions of derived equivalences}. \href{http:/ \! /arxiv.org/abs/1601.02030}{arXiv:1601.02030} (2016).

\bibitem[HHP09]{HHP}
Herbst, Manfred; Hori, Kentaro; Page, David.
{\em B-type D-branes in toric Calabi-Yau varieties. }Homological mirror symmetry, 
Lecture Notes in Phys., 757, Springer, Berlin, (2009), 27--44.

\bibitem[Hov01]{Hov01}
Hovey, Mark. {\em Model categories.}
Mathematical Surveys and Monographs, 63. American Mathematical Society, Providence, RI (1999), xii+209 pp.


\bibitem[Jar15]{Jardine}
Jardine, John F. Local homotopy theory. Springer Monographs in Mathematics. Springer, New York (2015), x+508 pp. 

\bibitem[Kaw04]{Kaw}
Kawamata, Yujiro. {\em D-equivalence and K-equivalence.} J . Differential Geom. 61 (2002), no. 1, 147--171. 

\bibitem[Kaw06]{KMukai}
Kawamata, Yujiro. {\em Derived equivalence for stratified Mukai flop on G(2,4)}. Mirror symmetry. V,  
AMS/IP Stud. Adv. Math., 38, Amer. Math. Soc., Providence, RI, (2006), 285--294.

\bibitem[Kaw08]{KFlops}
Kawamata, Yujiro. \emph{Flops Connect Minimal Models.} Publications of the Research Institute for Mathematical Sciences 44 (2008), 419--423.

\bibitem[Kra12]{Krause}
Krause, Henning. {\em Localization theory for triangulated categories.} Triangulated categories, 
London Math. Soc. Lecture Note Ser., 375, Cambridge Univ. Press, Cambridge (2010),  161--235.

\bibitem[Lun73]{Luna}
Luna, Domingo. {\em Slices \'etales.} Sur les groupes alg\'ebriques, Bull. Soc. Math. France, Paris, Mmoire 33 Soc. Math. France, Paris (1973),  81--105.

\bibitem[RMdS07]{RMdS}
Ruip\'erez, Daniel Hern\'andez, Ana Cristina L\'opez Martín, and Fernando Sancho de Salas. \emph{Fourier-Mukai transforms for Gorenstein schemes."} Advances in Mathematics 211.2 (2007), 594--620.

\bibitem[Nam03]{N1}
Namikawa, Yoshinori. 
{\em Mukai flops and derived categories.} J. Reine Angew. Math. 560 (2003), 65--76. 

\bibitem[Nam04]{N2}
Namikawa, Yoshinori. {\em Mukai flops and derived categories. II. } Algebraic structures and moduli spaces,  
CRM Proc. Lecture Notes, 38, Amer. Math. Soc., Providence, RI, (2004),  149--175.

\bibitem[Nee96]{Nee}{}
A. Neeman. \emph{The Grothendieck duality theorem via Bousfield'€™s techniques and Brown representability.} Journal of the American Mathematical Society 9.1 (1996), 205--236.

\bibitem[Orl92]{Orl92}
Orlov, Dmitry.  {\em Projective bundles, monoidal transformations, and derived categories of coherent sheaves.}
Izv. Ross. Akad. Nauk Ser. Mat. 56 (1992), no. 4, 852--862; translation in 
Russian Acad. Sci. Izv. Math. 41 (1993), no. 1, 133--141 

\bibitem[Orl09]{Orl09}{}
D. Orlov. \emph{Derived categories of coherent sheaves and triangulated categories of singularities}. Algebra, arithmetic, and geometry: in honor of Yu. I. Manin. Vol. II, Progr. Math., 270, Birkh\"{a}user Boston, Inc., Boston, MA (2009), 503--531.

\bibitem[Qui67]{Quillen}
Quillen, Daniel G. {\em Homotopical algebra.} Lecture Notes in Mathematics, No. 43 Springer-Verlag, Berlin-New York (1967), {\rm iv}+156 pp.

\bibitem[Sch97]{Schwede}
Schwede, Stefan. {\em Spectra in model categories and applications to the algebraic cotangent complex}. J. Pure Appl. Algebra 120 (1997), no. 1, 77--104. 

\bibitem[SS03]{SSmodules}
Schwede, Stefan; Shipley, Brooke.{\em Stable model categories are categories of modules.} Topology 42 (2003), no. 1, 103--153. 

\bibitem[Seg11]{Seg}
Segal, Ed. {\em Equivalence between GIT quotients of Landau-Ginzburg B-models. }
Comm. Math. Phys. 304 (2011), no. 2, 411--432.

\bibitem[SvdB17]{SvdBNon}
\v{S}penko, \v{S}pela; Van den Bergh, Michel. {\em Non-commutative resolutions of quotient singularities for reductive groups}. Invent. Math. 210 (2017), no. 1, 3--67.

\bibitem[SvdB16]{SvdB}
\v{S}penko, \v{S}pela; Van den Bergh, Michel. \emph{Semi-orthogonal decomposition of GIT quotient stacks.} \href{arXiv:1603.02858}{arXiv:1603.02858} (2016).


\bibitem[Tha96]{Thaddeus}
Thaddeus, Michael. {\em Geometric invariant theory and flips. } J. Amer. Math. Soc. 9 (1996), no. 3, 691--723. 

\bibitem[TV08]{TV}
To\"en, Bertrand;  Vezzosi, Gabrielle. {\em Brave new algebraic geometry and global derived moduli spaces of ring spectra.} Elliptic cohomology,  London Math. Soc. Lecture Note Ser., 342, Cambridge Univ. Press, Cambridge  (2007), 325--359. 

\end{thebibliography}
\end{document}